\documentclass[11pt,reqno]{amsart} 

\usepackage[utf8]{inputenc} 
\usepackage{bbm}

\usepackage[margin=1in]{geometry} 

\usepackage{graphicx} 

\usepackage{xcolor}

\usepackage{booktabs} 
\usepackage{array} 
\usepackage{paralist} 
\usepackage{verbatim} 
\usepackage{subfig} 
\usepackage{mathrsfs}
\usepackage{amssymb}
\usepackage{amsthm}
\usepackage{amsmath,amsfonts,amssymb,esint}
\usepackage{graphics,color}
\usepackage{enumerate}
\usepackage{mathtools,centernot}
\usepackage{cases}
\usepackage{cite, bookmark}
\usepackage{epsfig}
\usepackage{hyperref}
\usepackage{epstopdf}

\usepackage{xfrac}

\newcommand{\R}{\mathbb{R}}
\newcommand{\N}{\mathbb{N}}

\newcommand{\B}{\mathbb{B}}
\newcommand{\supp}{\textsf{supp}}

\numberwithin{equation}{section}

\newtheorem{theorem}{Theorem}[section]
\newtheorem*{theorem*}{Theorem}
\newtheorem{corollary}[theorem]{Corollary}
\newtheorem{proposition}[theorem]{Proposition}
\newtheorem{lemma}[theorem]{Lemma}

\theoremstyle{definition}
\newtheorem{definition}[theorem]{Definition}
\newtheorem{remark}[theorem]{Remark}
\newtheorem*{problem}{Problem}
\newtheorem{condition}[theorem]{Condition}

\title[isometric embedding]{On the isometric version of Whitney's strong embedding theorem}

\author{Wentao Cao}
\address{Academy for Multidisciplinary Studies, Capital Normal University, West 3rd Ring North Road 105, Beijing, 100048, P.R. China.}
\email{cwtmath@cnu.edu.cn}

\author{L\'aszl\'o Sz\'ekelyhidi Jr.}
\address{Institut f\"{u}r mathematik, Universit\"{a}t Leipzig, Augustusplatz 10, D-04109, Leipzig, Germany}
\email{laszlo.szekelyhidi@math.uni-leipzig.de}

\date{\today}


\begin{document}

\begin{abstract}
We prove a version of Whitney's strong embedding theorem for isometric embeddings within the general setting of the Nash-Kuiper h-principle. More precisely, we show that any $n$-dimensional smooth compact manifold admits infinitely many global isometric embeddings into $2n$-dimensional Euclidean space, of H\"older class $C^{1,\theta}$ with $\theta<1/3$ for $n=2$ and $\theta<(n+2)^{-1}$ for $n\geq3$. The proof is performed by Nash-Kuiper's convex integration construction and applying the gluing technique of the authors on short embeddings with small amplitude. 
\end{abstract}

\maketitle

%
\section{Introduction}

Differential manifold is generalised from Euclidean space. Naturally, one may wonder whether any abstract manifold can be embedded into some Euclidean spaces.
In 1944, Whitney \cite{Whitney1944} 
gave the answer and proved that
\begin{theorem*}[Whitney's Strong Embedding Theorem]
Any smooth $n$-dimensional compact manifold $\mathcal{M}^n$ can be embedded into $2n$-dimensional Euclidean spaces $\R^{2n}.$
\end{theorem*}
\noindent Furthermore, if the differential manifold $\mathcal{M}^n$  is imposed with some Riemannian metric $g$, can it still be embedded into the same Euclidean space maintaining the given metric? It is the main topic of the present paper.  The following problem is considered.

\begin{problem}\label{prob}
Let $(\mathcal{M}^n, g)$ be a smooth $n$-dimensional compact Riemannian manifold $\mathcal{M}^n$  with some metric $g$. Does there exist any isometric embedding $u:\mathcal{M}^n\hookrightarrow\R^{2n}$?
\end{problem}

\noindent Recall that a $C^1$ embedding $u:(\mathcal{M}^n, g)\hookrightarrow \R^m$ is isometric, provided the induced metric $u^\sharp e$ from the standard Euclidean metric $e$ on $\R^m$ agrees with $g$, i.e. $u^\sharp e$. In local coordinates $(x_1,\cdots, x_n)$, $g=\sum_{i,j=1}^ng_{ij}dx_idx_j$,  $(u^\sharp e)_{ij}=\partial_{x_i} u\cdot\partial_{x_j} u$ and $u=(u^1,\cdots, u^m)\in\R^m$, and the isometric embedding problem amounts to solving 
\begin{equation}\label{e:isometric}
\sum_{k=1}^m\frac{\partial u^k}{\partial x_i}\frac{\partial u^k}{\partial x_j}=g_{ij}, \quad i,j=1,\cdots, n,
\end{equation}
which is a system of $n_*:={n(n+1)}/2$ equations.

\subsection{Background}

If the manifold and the metric is real analytic, M.~Janet \cite{Janet1927} showed that any $(\mathcal{M}^2,g)$ admits locally an analytic isometric embedding into $\R^3$. Sunsequently E.~Cartan \cite{Cartan1927} extended Janet's result to local analytic isometric embeddings $(\mathcal{M}^n,g)\hookrightarrow \R^{n_*}$. Indeed, noting that the (local) PDE \eqref{e:isometric} consists of $n_*$ equations, the target space has to have the same dimension $m=n_*$ in order to have a formally well-defined system of equations. 

If the metric is smooth, the situation changes completely. Indeed, the analogue of Janet's local result $(\mathcal{M}^2,g)\hookrightarrow \R^3$ in the smooth category in full generality is still an outstanding open problem (however, for results involving additional conditions on the curvature, we refer to \cite{HanHongBook}). On the other hand, starting with the startling discoveries of Nash \cite{Nash:1954vt,Nash:1956ty} and later molded into a very powerful and far-reaching principle by Gromov \cite{Gromov1973,Gromov,Gromov:2015tua}, if either the regularity of the embedding or the target dimension $m$ is substantially relaxed, one can recover existence at the expense of losing uniqueness - this is the essence of the h-principle. More precisely, 
\begin{enumerate}
    \item[(A)] In \cite{Nash:1956ty} Nash  showed that if the target dimension $m\geq 3n_*+4n$, then any short immersion can be uniformly approximated by isometric immersions of class $C^{\infty}$ 
    \item[(B)] In \cite{Nash:1954vt} he showed that if $m\geq n+2$, then any short immersion can be uniformly approximated by isometric immersions of class $C^1$. The latter was subsequently extended to $m\geq n+1$ by Kuiper \cite{Kuiper:1955th}.
\end{enumerate}
In both cases, if the initial short map is an embedding, so are the approximating isometric maps. 

Optimality of these results is a problem of current interest. In setting (A) the question is about optimality of the target dimension $m$. In \cite{GroRoh1970,Gromov} Nash's result was extended to $m\geq n_*+2n+3$, which was further improved by M.G\"unther in \cite{Gunther1989,Gunther1991} to $m\geq n_*+\max\{2n,n+5\}$. It is not known what is the optimal (i.e.~smallest) target dimension. 

In setting (B) the question is about the optimality of the regularity of the immersion, in particular on the H\"older scale $C^{1,\theta}$. Yu.~Borisov \cite{Borisov:1965wf,Borisov:2004wo} and subsequently Conti-De Lellis and the second author \cite{CoDeSz2012,Szek2014:LN} showed that the Nash-Kuiper theorem also holds for locally $C^{1, \theta}$ isometric immersions $u:\mathcal{M}^n\hookrightarrow\R^{m}(m\geq n+1)$ with $\theta<(1+2n_*)^{-1}$, which can be extended to global isometric immersions with $\theta<(1+2(n+1)n_*)^{-1}$. Moreover, for 2-dimensional discs the exponent above can be improved from $1/7$ to $1/5$  \cite{DISz2018}. In \cite{CaoSze2022} the authors were recently able to devise an induction-on-dimension scheme that, combined with the local convex integration iterations in \cite{CoDeSz2012,DISz2018}, are able to extend the local constructions to compact manifolds. Thus, the current state of the art for isometric immersions $u:(\mathcal{M},g)\hookrightarrow \R^{n+1}$ is the statement (B) above with $u\in C^{1,\theta}$, $\theta<(1+2n_*)^{-1}$ for $n\geq 3$ and $\theta<1/5$ for $n=2$. As far as the \emph{existence} of isometric embeddings is concerned, in combination with Whitney's embedding theorem this yields the existence of $C^{1,\theta}$ isometric embeddings $(\mathcal{M},g)\hookrightarrow\R^{2n}$. Our aim in this paper is to improve the threshold H\"older exponent in this particular setting.

The Nash-Kuiper iteration has been extended further in the case the target dimension is sufficiently large $m\geq 6(n+1)(n_*+1)+2n$ and the metric is non-smooth, $g\in C^{\beta}$, $\beta<2$, (in which case the techniques of \cite{Nash:1956ty,GroRoh1970,Gunther1991} do not apply), to yield $C^{1\theta}$ isometric embeddings for any $\theta<1+\beta/2$, and in the local setting of 2D polar caps De Lellis-Inauen have constructed $C^{1, \theta}(\theta<1/2)$ isometric embeddings into $\R^{14}$ in \cite{DI2020}. This construction was recently extended in \cite{CaoIn2020} to $C^{1, \theta}(\theta<1/2)$ isometric embeddings of general $\mathcal{M}^n$ into $\R^{n+2n_*}$. But the dimensions of the target spaces in these results exceeds $2n$.

\subsection{Main result and strategy}

The rule of thumb of the threshold H\"older exponent in local versions of the Nash-Kuiper convex integration scheme is as follows: the metric error is removed iteratively in stages. Each stage consists of several (fixed finite number, say, $N$) steps, and each step consists of adding a high-frequency perturbation (e.g. a spiral as in \cite{Nash:1954vt} or a corrugation as in \cite{Kuiper:1955th}, \cite[Lemma 4.4]{CoDeSz2012} or \cite[Proposition 2.4]{DISz2018}). The number of steps in a stage will determine the threshold H\"older exponent: $\theta_c=(1+2N)^{-1}$. This is explained in detail in e.g.  \cite[Section 3]{CoDeSz2012}. Following the observation that simple spirals/corrugations are able to remove a rank-one component of the metric error (called a \emph{primitive metric}), Nash used a decomposition of the metric error into a sum of primitive metrics \cite[Lemma 1]{Nash:1954vt} (see also \cite[Lemma 5.2]{CoDeSz2012} and Lemma \ref{le-decompose} below), one sees that a natural guess for the number of steps $N$ in low co-dimension ($m=n+1$) is $N=n_*$, the dimension of the space of $n\times n$ symmetric matrices. This leads to the H\"older exponent $\theta<(1+2n_*)^{-1}$. More generally, if the codimension is higher, e.g.~$m=n+k$, one can hope to remove $k$ primitive metrics in a single step, leading to $N=\lceil\tfrac{n_*}{k}\rceil$ steps. This idea was observed already by K\"allen \cite{Kallen1978}, using Nash spirals instead of corrugations, and recently also in the context of the (closely related) Monge-Ampere equation by M.~Lewicka \cite{lewicka2022}. In particular, in our setting $k=n$, and indeed, our main result is to achieve the threshold H\"older exponent $\theta<(1+2n_*/n)^{-1}=(2+n)^{-1}$ in the context of global isometric embeddings. 

\begin{theorem}[Main result] \label{th-n-2n}
Let $(\mathcal{M}^n,g)$ be any $n$-dimensional smooth compact manifold with $C^{1}$ metric. For any $\theta<\theta(n)$ with
\[\displaystyle\theta(n):=\begin{cases} \dfrac13,\,\, n=2,\\
\dfrac1{n+2},\,\, n\geq3,\end{cases}\]
there exist infinitely many $C^{1, \theta}$ isometric embeddings of $(\mathcal{M}^n, g)$ into $\R^{2n}$.
\end{theorem}

Our strategy for proving Theorem \ref{th-n-2n} follows the basic framework developed in \cite{CoDeSz2012} together with the local-to-global scheme via adapted short maps developed in \cite{CaoSze2022}. 

We remark that the ``rule of thumb'' explained above indeed works in a rigorous implementation, with minor modifications compared to \cite{CoDeSz2012,CaoSze2022} \emph{in the case when $n$ is odd}, because in this case $n_*/n$ is an integer. Furthermore, the case $n=2$ again relies on the idea involving conformal coordinates introduced in \cite{DISz2018}. We also remark in passing that using a very similar strategy one can also show, for any $n\geq 3$, the existence of $C^{1,\theta}$-isometric embeddings $u:(\mathcal{M}^n,g)\rightharpoonup \R^{n+n_*}$ for any $\theta<1/3$. 

However, if $n\geq 4$ is even, $n_*/n$ is not an integer, and using the naive strategy sketched above would lead to $N=\lceil\tfrac{n_*}{n}\rceil=n/2+1$ and the lower H\"older exponent $(3+n)^{-1}$. Instead, we use an idea, originating from K\"allen \cite{Kallen1978} and used in \cite{DI2020,CaoIn2020}, of absorbing certain higher order error terms in the metric decomposition at each stage, to obtain the H\"older exponent $(2+n)^{-1}$.

The rest of the paper is organized as follows. Outline of the construction is given in Section \ref{se-outline}, where  Theorem \ref{th-n-2n} will be proved assuming the main Propositions  \ref{pr-stage} and \ref{pr-inductive}. Section \ref{se-stage} and Section \ref{se-inductive} are devoted to show Proposition \ref{pr-stage} and Proposition \ref{pr-inductive} respectively.
\section{Outline of the proof}\label{se-outline}

The proof of Theorem \ref{th-n-2n} proceeds using a variant of the iteration scheme introduced by Nash and Kuiper \cite{Nash:1954vt,Kuiper:1955th} for the construction of $C^1$ isometric embeddings, more precisely on the quantitative version of the iteration as it has been developed in a series of works \cite{Kallen1978, Borisov:2004wo, CoDeSz2012, DISz2018,CaoSze2019, DI2020, CaoSze2022}. In particular, we will follow the strategy developed in \cite{CaoSze2022} for extending the local H\"older scheme to global embeddings. 

\subsection{Local corrections of the metric}
The heart of the matter of the Nash-Kuiper scheme is an iteration procedure which corrects the metric error in a single chart by adding high-frequency perturbations (``spirals'' in \cite{Nash:1954vt}, ''corrugations'' in \cite{Kuiper:1955th}). The key point is that addition of a single corrugation can correct the metric error (upto a new error of arbitrarily small size) in a rank-one direction - called a ``primitive metric''.  The basic decomposition, based on Nash's work, involves $n_*=n(n+1)/2$ terms:
\begin{lemma}\label{le-decompose}
Let $P_0\in \R^{n\times n}_{sym,+}$. There exists a constant $\sigma_0=\sigma_0(n)>0$ and vectors $\xi_1,\dots,\xi_{n_*}\in \mathbb S^{n-1}$ and linear maps $L_i:\R^{n\times n}_{sym}\to\R$ such that 
$$
P=\sum_{i=1}^{n_*}L_i(P)\xi_i\otimes \xi_i
$$
and moreover $L_i(P)$ is bounded from below for every $i$ and for every $P\in \R^{n\times n}_{sym}$ with $|P-P_0|\leq \sigma_0$.
\end{lemma}
We will use certain variants of this decomposition, see Lemma \ref{le-perturb}.

When $n=2$, we can decompose the metric error into only two primitive metrics.
\begin{lemma}\label{le-conformal}
Let $\Omega\subset\R^2$ be a simply connected open bounded set with smooth boundary and $P:\Omega\to \R^{2\times 2}_{sym,+}$ such that, for some $0<\alpha<1$ and $\gamma,M\geq 1$
\begin{equation}\label{eq-p-condition}
{\gamma}^{-1}\mathrm{Id}\leq P\leq \gamma\mathrm{Id},\quad \|P\|_{C^\alpha(\Omega)}\leq M.
\end{equation}
Then there exists a smooth diffeomorphism $\Phi:\Omega\to\R^2$ and a smooth positive function $a:\Omega\to\R$ satisfying
\begin{equation}\label{eq-isothermal}
P=a^2(\nabla\Phi_1\otimes\nabla\Phi_1+\nabla\Phi_2\otimes\nabla\Phi_2),
\end{equation}
and the following estimates
\begin{equation*}
\begin{split}
&\det(D\Phi(x))\geq C_0^{-1},\quad a(x)\geq C_0^{-1}\quad \textrm{ for all }x\in\Omega,\\
&\|a\|_{C^{j,\alpha}(\Omega)}+\|\nabla\Phi\|_{C^{j,\alpha}(\Omega)}\leq C_j\|P\|_{C^{j, \alpha}(\Omega)}, j\in\mathbb{N},
\end{split}
\end{equation*}
where the constants $C_j\geq 1$ depend only on $\alpha, \gamma, M$ and on $\Omega$.
\end{lemma}
After decomposing the metric error into a sum of primitive metrics, we can remove each primitive metric error term by adding corrugations in an appropriate number of successive steps (c.f.~\cite[Corollary 3.1 and 3.2]{CaoSze2022}).

\begin{proposition}[Stage]\label{pr-stage}
There exists $\sigma_0=\sigma_0(n)>0$ with the following properties. Let $G$ be a $C^1$ metric on $\Omega\subset\R^n$ with ${\gamma}^{-1}\mathrm{Id}\leq G\leq \gamma\mathrm{Id}$, $\|G\|_1\leq \gamma$ for some $\gamma\geq 1$ and such that $\textsf{osc}_\Omega G\leq \sigma_0$. Let the embedding $u\in C^2(\Omega; \R^{2n})$, $\rho\in C^1(\Omega)$, and $H\in C^1(\Omega;\R^{n\times n}_{sym})$ satisfy
\begin{equation}\label{eq-stage-u}
{\gamma}^{-1}\mathrm{Id}\leq \nabla u^T\nabla u\leq\gamma \mathrm{Id}\,\textrm{ in }\Omega,\quad
\|u\|_{2}\leq \delta^{1/2}\lambda,
\end{equation}
and
\begin{align}
\|\rho\|_0\leq \delta^{1/2}, &\quad \|\rho\|_1\leq \delta^{1/2}\lambda, \label{eq-stage-rho}\\ 
\|H\|_0\leq \lambda^{-\alpha}, &\quad \|H\|_1\leq \lambda^{1-\alpha}\,, \label{eq-stage-H}
\end{align}
for some $0<\delta<1$, $\alpha>0$ and $\lambda>1$. Then, for any $\kappa> 1$ there exist constants $\delta_*=\delta_*(\gamma, n)>0$ and $\lambda_*=\lambda_*(\kappa, \alpha, \gamma, n)>1$ such that if
\begin{equation}\label{eq-stage-lambda0}
\delta\leq\delta_*, \quad \lambda\geq\lambda_*,
\end{equation}
then, there exists a new embedding $v\in C^2(\Omega; \R^{2n})$ and $\mathcal{E}\in C^1(\Omega;\R^{n\times n}_{sym})$ such that
\begin{equation}\label{eq-stage-1}
\begin{split}
 \nabla v^T\nabla v=&\nabla u^T\nabla u+\rho^2(G+H)+\mathcal{E}\quad\textrm{ in }\Omega,\\
v=&u\textrm{ on }\Omega\setminus(\supp\rho+\B_{\lambda^{-\kappa}}(0))
\end{split}
\end{equation}
with estimates
\begin{align}
\|v-u\|_0\leq C\delta^{1/2}\lambda^{-(\kappa+1)/2}, \quad 
\|v-u\|_1\leq C\delta^{1/2}, \quad
\|v\|_2\leq C\delta^{1/2}\lambda^{N(\kappa-1)+1}, \label{eq-stage-2}
\end{align}
and
\begin{equation}\label{eq-stage-3}
\|\mathcal{E}\|_0\leq C(\delta\lambda^{1-\kappa}+\delta^2), \quad
\|\mathcal{E}\|_1\leq C\left(\delta\lambda^{(N-1)(\kappa-1)+1}+
\delta^2\lambda^{N(\kappa-1)+1}\right),
\end{equation}
where $N=\frac{1-\theta(n)}{2\theta(n)}$.
The constant $C\geq 1$ depends only on $\alpha$ and $\gamma$.
\end{proposition}
\noindent Proposition \ref{pr-stage}, with the choice  $N=\frac{1-\theta(n)}{2\theta(n)}$, is the key ingredient for obtaining the threshold H\"older exponent $\theta(n)$ in Theorem \ref{th-n-2n}, and is, apart from the different choice of $N$, a close analogue of Corollaries 3.1--3.2 in \cite{CaoSze2022}. Indeed, a close inspection shows that the only other difference appears in the additional $\delta^2$ terms in \eqref{eq-stage-3}, but it turns out that this additional term is lower order. We remark in passing that, even though in our introduction $N$ was motivated as the number of steps in a stage, here $N$ does not need to be an integer. The proof of Proposition \ref{pr-stage} will be given in Section \ref{se-stage} and is the main new contribution of this work.

\subsection{Global Setup}\label{se-setup}

Throughout the paper we consider a compact Riemannian manifold $(\mathcal{M}, g)$ with a $C^1$ metric $g$. We fix a finite atlas $\{\Omega_k\}_k$ of $\mathcal{M}$ with charts $\Omega_k$ and a corresponding partition of unity $\{\phi_k\}$ so that
$$
\sum \phi_k^2=1 \text{ and }\phi_k\in C_c^{\infty}(\Omega_k).
$$
Furthermore, on each $\Omega_k$ we fix a choice of coordinates and in this way identify $\Omega_k$ with a bounded open subset of $\R^n$. We write $g=(G_{ij})$ in any local chart $\Omega_k$ and the same way denote any symmetric 2-tensor $h$ in local charts by $(H_{ij})$. Then, for any $x\in\Omega_k$, $G(x),H(x)\in \R^{n\times n}_{sym}$, the set of $n\times n$ symmetric matrices. In particular, if $u:\mathcal{M}\to\R^m$ is a $C^1$ map, the induced (pullback) metric on $\mathcal{M}$, denoted by $u^\sharp e$, can be written as $\nabla u^T\nabla u$ in any local chart. 

Observe that, since $g$ is a (non-degenerate) metric of class $C^1$ and $\mathcal{M}$ is compact, there exists a constant $\gamma\geq 1$ such that, in any local chart $\Omega_k$, we have
\begin{equation}\label{e:gamma0}
{\gamma}^{-1}\textrm{Id}\leq G\leq\gamma\textrm{Id},\quad \|G\|_{C^1(\Omega_k)}\leq \gamma.
\end{equation}
By refining the charts if necessary, we may ensure in addition that 
\begin{equation}\label{e:oscillation}
\textrm{osc}_{\Omega_k}G<\frac{1}{4}\sigma_0,
\end{equation}
where $\sigma_0>0$ is the dimensional constant in Lemma \ref{le-decompose}.

As usual, we define the supremum norm of maps $f:\mathcal{M}\to\R^m$ as $\|f\|_0=\sup_{x\in\mathcal{M}}|f(x)|$. Furthermore, we use the given atlas and associated partition of unity to define the H\"older norms: for any $r\geq 0$ we set
$$
[u]_r:=\sum_k[\phi_k^2u]_r.
$$
Similarly, we define ``mollification on $\mathcal{M}$"  through the partition of unity. That is to say, for a function $u$ on $\mathcal{M}$ we define
$$
u*\varphi_\ell=\sum_k (u\phi_k^2)*\varphi_\ell.
$$
Finally, we note that these definitions can be easily extended to symmetric 2-tensors $h$ on $\mathcal{M}$ using the pointwise norm given by the underlying metric $g$:
$$
|h(x)|=\sup_{\xi\in T_x\mathcal{M}, |\xi|_g=1}|h(\xi,\xi)|,
$$
where $|\xi|_g=(\sum_{ij}g_{ij}\xi_i\xi_j)^{1/2}$.
Note that because of \eqref{e:gamma0} this norm is equivalent to the matrix norm of $H(x)$ given by
$$
|H(x)|=\sup_{|\xi|=1}|H_{ij}(x)\xi_i\xi_j|
$$
In particular, given the $C^1$ metric $g$ on $\mathcal{M}$ (with the local representation $g=(G_{ij})$ in local charts $\Omega_k$), we may choose $\gamma_0$ from \eqref{e:gamma0} sufficiently large so that in addition
\begin{equation}\label{e:gamma1}
	\|g\|_1\leq \gamma_0.
\end{equation}

\subsection{Passing from local to global corrections}

In \cite{CaoSze2022} we introduced a procedure for combining the local convex integration steps into a global iteration without loss of regularity. It is based on the induction on the dimension of the skeleta of a given triangulation of the manifold $\mathcal{M}$ and carefully controlling the regularity of the approximating maps using the size of the metric error. We start by recalling the main concepts. 

First of all we recall \cite{Cairns1961} that any smooth compact manifold $(\mathcal{M}^n,g)$ admits a regular triangulation $\mathcal{T}$, where skeleta of $\mathcal{T}$ is composed with a finite union of $C^1$ submanifolds, and each simplex $T\in\mathcal{T}$ belongs to a single chart $\Omega_k$ for some $k$. Moreover, by compactness, the following regularity condition holds, whenever $S\subset \Sigma$ are the skeleta in $\mathcal{T}$ of consecutive dimension:
\begin{condition}\label{c:geometric}
There exists a geometric constant $\bar{r}>0$ such that
for any $\delta>0$ the set 
$$
\biggl\{x\in\mathcal{M}:\,\textrm{dist}(x,S)\geq \delta\textrm{ and }\textrm{dist}(x,\Sigma)\leq \bar{r}\delta\biggr\}
$$
is contained in a pairwise disjoint union of open sets, each contained in a single chart 
$\Omega_k$.
\end{condition}

In order to obtain an isometric map via convex integration with the required H\"older regularity, we need to keep track of the $C^2$ norm the approximating maps (c.f.~\eqref{eq-stage-u}) as well as the $C^1$ norm of the metric error (c.f.~\eqref{eq-stage-rho}-\eqref{eq-stage-H}). For maps with metric error which may be vanishing on some subset, the norms need to be related to the size of the metric error. This is formalized in the concept of adapted short embedding, introduced in \cite{DaneriSz2016} in the context of the Euler equations and in \cite{CaoSze2022} in the context of isometric immersions. Our definition here is slightly different from \cite[Definition 2.3]{CaoSze2022} in that we include additional smallness assumptions on the metric error using the exponents $\alpha,\beta$.
\begin{definition}\label{de-adapt} (\textit{Adapted short embedding})
 A short embedding $u:(\mathcal{M}^n, g)\hookrightarrow\R^m$ is called \emph{adapted short embedding with respect to $\Sigma$ with parameters $(\theta, \beta, \alpha, A)$} where $\theta, \beta, \alpha\in(0, 1), A\geq1$, 
 if $u\in C^{1, \theta}(\mathcal{M}^n)$ satisfies
 \begin{equation}\label{eq-adapt-1}
 g-u^\sharp e=\rho^2(g+h)
 \end{equation}
 with $\Sigma=\{\rho=0\}$, $u\in C^2(\mathcal{M}^n\setminus \Sigma),\,\rho,h\in C^1(\mathcal{M}^n\setminus\Sigma)$ and the following estimates
\begin{equation}\label{eq-adapt}
\begin{split}
&|\nabla^2u(x)|\leq A\rho(x)^{1-\tfrac{1}{\theta}},\quad 0<\rho(x)\leq A^{-\beta}, \quad |h(x)|\leq A^{-\alpha\theta}\rho(x)^\alpha,\\
&\quad\quad |\nabla\rho(x)|\leq A\rho(x)^{1-\tfrac{1}{\theta}},\quad |\nabla h(x)|\leq A^{1-\alpha\theta}\rho(x)^{\alpha-\tfrac{1}{\theta}},
\end{split}
\end{equation}
for any $x\in\Omega_k\setminus \Sigma$. Here $\Omega_k$ is some chart of $\mathcal{M}^n$.
\end{definition}

Based on Proposition \ref{pr-stage} we have the following result, which is a variant of \cite[Proposition 4.1]{CaoSze2022}.

\begin{proposition}\label{pr-inductive}
Let $S\subset\Sigma$ be skeleta of the triangulation $\mathcal{T}$ of consecutive dimension satisfying condition \ref{c:geometric}. Let $$0<\theta<\frac{1}{1+2N},\quad 0<\beta<1.$$
There exists $c_*(N,\theta)>0$ such that, for any  
\begin{equation}\label{e:alpha*beta}
0<\alpha<c_*\beta 
\end{equation}
there exists a positive constant $A_*(N,\theta,\beta,\alpha)>1$ such that the following holds:

Let $u$ be an adapted short embedding $\mathcal{M}^n\hookrightarrow\R^{2n}$ with respect to $S$ and parameters $(\theta, \beta, \alpha, A).$  For any $A\geq A_*$, then there exists a new adapted short embedding $\bar u$ with respect to $\Sigma\supset S$ and parameters $(\theta', \beta', \alpha', A')$ with
\begin{equation}\label{e:choiceb}
\theta'=\frac{\theta}{b^2}, \quad \beta'=\frac\beta{b^2}, \quad  \alpha'=\frac{\alpha}{2b^2}, \quad A'=A^{b^2}, \quad \text{ where } b=1+\frac{2N\theta\alpha}{1-\theta(1+2N)}.
\end{equation}
Furthermore $\bar u$ satisfies $\|\bar u-u\|_0\leq A^{-1}, \, \bar\rho\le\rho$ and $\bar u=u$ on $S$.

\end{proposition}

\begin{remark}
The proof of Proposition \ref{pr-inductive} follows closely the proof of \cite[Proposition 4.1]{CaoSze2022}. The key difference is that in Proposition \ref{pr-stage} we have estimates \eqref{eq-stage-2}-\eqref{eq-stage-3} with $N=\frac{1-\theta(n)}{2\theta(n)}$, whereas in \cite{CaoSze2022} the corresponding local stage produced estimates with $N=2$ for $n=2$ \cite[Corollary 3.1]{CaoSze2022} and $N=n_*$ for $n\geq 3$ \cite[Corollary 3.2]{CaoSze2022} (compare (4.34) therein to the choice of $b$ in \eqref{e:choiceb}), and in our current setting there is an additional error term in \eqref{eq-stage-3}. On the other hand, our proof below makes apparent that the analogue of Proposition \ref{pr-inductive} as well as of Theorem \ref{th-n-2n} will hold, provided the statement of Proposition \ref{pr-stage} holds with a different value of $N$ (e.g. for general codimension). 
\end{remark}

\subsection{The proof of Theorem \ref{th-n-2n}}

The proof of Theorem \ref{th-n-2n} proceeds via an induction on the dimension of skeleta of the triangulation $\mathcal{T}$ defined in Section \ref{se-setup} above, using Proposition \ref{pr-inductive} as the inductive step.

By compactness of $\mathcal{M}$ and a simple mollification argument we may assume that the given short embedding $u$ is smooth and strictly short.
For the purpose of obtaining an initial adapted short embedding with metric error of small amplitude from a given strictly short embedding, we use a slightly modified version of Proposition 5.1 from \cite{CaoSze2022}.
\begin{proposition}\label{p:initial}
Let $u\in C^2(\mathcal{M};\R^m)$ be a strictly short immersion. There exist constants $C_*=C_*(g,u)\geq 1$, $A_*=A_*(g,u)\geq 1$ and $\alpha_*=\alpha_*(g,u)>0$ such that, for any $A\geq A_*$ and $\alpha<\alpha_*$ there exists a strongly short immersion $\tilde u\in C^2(\mathcal{M};\R^m)$ with associated $\tilde \rho,\tilde h$ such that
\begin{equation}\label{e:strong-1}
    g-\tilde u^\sharp e=\tilde\rho^2(g+\tilde h)
\end{equation}
with
\begin{equation}\label{e:strong-rho}
      \tilde\rho=C_*A^{(\alpha-\alpha_*)/2}  
\end{equation}
and moreover the following estimates hold:
\begin{align}
    \|\tilde u-u\|_0&\leq \tilde\rho^2A^{-\alpha},\quad \|\tilde u\|_2\leq A,\label{e:strong-u}\\
    \|\tilde h\|_0&\leq A^{-\alpha},\quad \|\tilde h\|_1\leq A^{1-\alpha}.\label{e:strong-h}
\end{align}
\end{proposition}

\begin{proof}[Proof of Proposition \ref{p:initial}]
Since $u$ is strictly short and $\mathcal{M}$ is compact, there exists $\delta^*>0$ such that
\begin{align*}
g-u^\sharp e\geq 2\delta^* g.
\end{align*}
Consequently, for any $\delta\in (0,\delta_*)$ we have $g-u^\sharp e-\delta g\geq \delta^* g$ and hence we may apply 
Lemma 1 in \cite{Nash:1954vt} (see also Lemma 1 in \cite{Szek2014:LN}) to obtain $N_*=N_*(\mathcal{M},\delta_*, g)\in \N$ and a decomposion of the metric error into finite number of primitive metrics in the different charts $\Omega_k$ as
\begin{equation}
\label{e:strong-decompose}
g-u^\sharp e-\delta g=\sum_{j=1}^{N_*}a_j(x)^2\xi_j\otimes\xi_j,
\end{equation}
where $|\xi_j|=1$ and each $a_j\in C^1$ is supported in a single chart $\Omega_k$ for some $k$. 
Utilizing \cite[Proposition 3.2]{CaoSze2022} or \cite[Proposition 3.2]{CaoSze2019} we obtain, for any $\Lambda\geq \Lambda_*=\Lambda_*(u,g)$ an immersion $\tilde{u}$ and metric error $\mathcal{E}$ such that
\begin{align*}
\tilde{u}^\sharp e=u^\sharp e+\sum^{N_*}_{j}a_j(x)^2\xi_j\otimes\xi_j+\mathcal{E}
\end{align*}
with
\begin{equation}
\label{e:strong-error}
\begin{split}
\|\tilde{u}-u\|_0\leq\frac{C_1}{\Lambda}, &\qquad \|\tilde{u}\|_2\leq C_1\Lambda^{N_*},\\
\|\mathcal{E}\|_0\leq\frac{C_1}{\Lambda}, &\qquad \|\mathcal{E}\|_1\leq C_1\Lambda^{N_*-1},
\end{split}
\end{equation}
where $C_1\geq 1$ is a constant depending only on $u, g$. Thus, using \eqref{e:strong-decompose}, we have
\begin{align*}
g-\tilde{u}^\sharp e=\delta \left(g-\frac{\mathcal{E}}{\delta}\right),
\end{align*}
so that \eqref{e:strong-1} holds with 
$$\tilde{h}=-\frac{\mathcal{E}}{\delta},\quad \tilde\rho=\delta^{1/2}.$$
Now set
$$
\Lambda=C_1^{-1/N_*}A^{1/N_*},\quad \delta=C_1^{1+1/N_*}A^{\alpha-1/N_*}
$$
with $0<\alpha<1/N_*$. Then, using \eqref{e:strong-error} we deduce
\begin{equation*}
    \|\tilde u-u\|_0\leq C_1^{1+1/N_*} A^{-1/N_*}=\delta A^{-\alpha},\quad \|\tilde u\|_{2}\leq A, 
\end{equation*}
and
\begin{equation*}
    \|\tilde h\|_0\leq A^{-\alpha},\quad \|\tilde h\|_{1}\leq A^{1-\alpha}.
\end{equation*}
This proves the statement of the Proposition, with $\alpha_*=1/(2N_*)$, $C_*=C_1^{(1+N_*)/(2N_*)}$ and $A\geq A_*$ sufficiently large so that $\delta<\delta_*$ and $\Lambda\geq \Lambda_*$, depending on $g,u$ as well as $C_1$.

\end{proof}

In particular we have the following consequence:
\begin{corollary}\label{c:initial}
Let $u\in C^2(\mathcal{M};\R^m)$ be a strictly short immersion. There exists $\beta_*=\beta_*(u,g)>0$ such that, for any $\beta_0<\beta_*$ and any $0<\theta_0\leq 1$ there exists $A_*=A_*(u,g,\beta_0,\theta_0)\geq 1$ with the following property: for any $A_0\geq A_*$ there exists $\alpha_0>0$ and an adapted short immersion $\tilde u\in C^2(\mathcal{M};\R^m)$ with respect to the empty set $\Sigma=\emptyset$ with parameters $(\theta_0,\beta_0,\alpha_0,A_0)$. Moreover, $\|\tilde u-u\|_{0}\leq A_0^{-\alpha}$.
\end{corollary}

We remark that being adapted with respect to $\Sigma=\emptyset$ amounts to the condition that inequalities \eqref{eq-adapt} in Definition \ref{de-adapt} hold globally in $\mathcal{M}$.

\begin{proof}[Proof of Corollary \ref{c:initial}]
We aim to apply Proposition \ref{p:initial} to obtain $\tilde u\in C^2(\mathcal{M};\R^m)$ with properties \eqref{e:strong-1}-\eqref{e:strong-h} with $A=A_0$ to be fixed. Define $\beta_*=\alpha_*/2$, where $\alpha_*(g,u)$ is given in the proposition. For any $\beta_0<\beta_*$ fix $0<\alpha<\alpha_*$ such that $\alpha<\alpha_*-2\beta_0$. Then, choose $\alpha_0>0$ so that 
\begin{equation}\label{e:choicealpha0}
\alpha_0\left(\theta_0+\frac12(\alpha_*-\alpha)\right)< \alpha.
\end{equation}
Then, from \eqref{e:strong-rho} we obtain
$$
0<\tilde\rho=C_*A_0^{(\alpha-\alpha_*)/2}\leq A_0^{-\beta_0}
$$
for sufficiently large $A_0\geq A_*$ (depending on $C_*$ as well as on $(\alpha_*-\alpha)/2-\beta_0>0$), and obviously $\nabla\tilde\rho=0$. Further, by \eqref{e:strong-u}, we have
\[\|u-\tilde u\|_0\leq A_0^{-\alpha},\quad \|\nabla^2\tilde u|\leq A_0\leq A_0\tilde\rho^{1-1/\theta_0}\]
since $\tilde\rho\leq 1$ and $\theta_0\in (0,1]$. Finally, by \eqref{e:strong-h}, we get
\begin{align*}
    |h(x)|&\leq A_0^{-\alpha}\leq A_0^{-\alpha_0\theta_0}\tilde\rho^{\alpha_0},\\
    |\nabla h(x)|&\leq A_0^{1-\alpha}\leq A_0^{1-\alpha_0\theta_0}\tilde\rho^{\alpha_0-1/\theta_0}
\end{align*}
because of \eqref{e:choicealpha0}. Thus, we verified all the conditions of Definition \ref{de-adapt}, concluding the proof.
\end{proof}

Therefore, we utilize Corollary \ref{c:initial} to obtain the initial adapted short immersion $u_0=\tilde u$. Since $u$ is an embedding, the estimate $\|u_0-u\|_0\leq A_0^{-\alpha}$ ensures uniform closeness from which it follows that $u_0$ is an embedding provided $A_0$ is sufficiently large. 
We conclude that for such choice of $\alpha_0,\beta_0$ the initial map $u_0$ is an adapted strictly short embedding with parameters $(\theta_0,\beta_0,\alpha_0,A_0)$.

Now we can proceed by induction on the dimension of skeleta, analogously to \cite{CaoSze2022}. 

First of all, as a general remark, observe that if $u$ is an adapted short embedding satisfying \eqref{eq-adapt} with parameters $(\theta,\beta,\alpha,A)$, then it is also satisfies \eqref{eq-adapt} with parameters $(\theta',\beta',\alpha',A)$ for any $\theta'<\theta$, $\beta'<\beta$ and $\alpha'<\alpha$. 

Our aim is to construct a sequence of short embeddings $u_1,u_2,\dots,u_{n+1}$ which are adapted with respect to the skeleta $\Sigma_1\subset\Sigma_2\subset\dots\Sigma_{n+1}=\mathcal{M}^n$ with parameters $(\theta_j,\beta_j,\alpha_j,A_j)$ where
\begin{equation}\label{e:parametersj}
A_{j+1}=A_j^{b_j^2}, \quad \theta_{j+1}=\frac{\theta_j}{b_j^2}, \quad  \beta_{j+1}=\frac{\beta_j}{b_j^2},\quad \alpha_{j+1}=\frac{\alpha_j}{2b_j^2},
\end{equation}
where $$b_j=1+\frac{2N\alpha_j\theta_j}{1-\theta_j(1+2N)}.$$ 

For any $\theta<(1+2N)^{-1}=\theta(n)$ (recall $N$ is determined in Proposition \ref{pr-stage}) choose $\theta_0$ with $\theta< \theta_0<(1+2N)^{-1}$. From the recursion \eqref{e:parametersj} we see that $\theta_{n+1}<\theta_n<\dots<\theta_1<\theta_0$ with $$\theta_{n+1}=(\prod_{j=0}^nb_j^{-2})\theta_0,$$
with $\alpha_{n}<\dots<\alpha_0$ and $1<b_{n}<\dots<b_0$ also recursively defined. We then see that we can choose $\alpha_0>0$ sufficiently small so that $\theta_{n+1}>\theta$ and moreover that $\alpha_j<c_*(N,\theta_j)\beta_j$ (c.f.~\eqref{e:alpha*beta}). We can then choose $A_0\geq 1$ sufficiently large so that in addition
$A_j\geq A_*(N,\theta_j,\beta_j,\alpha_j)$ for all $j=0,1,\dots,n$. 

With these choices of parameters we first apply Corollary \ref{c:initial} to obtain $u_0$ and then apply Proposition \ref{pr-inductive} $n+1$ times to obtain $u_1,\dots,u_{n+1}$. By choosing $A_0$ even larger if necessary we ensure in addition $\|u_{n+1}-u\|_0$ is sufficiently small so that $u_{n+1}$ is an embedding, which then is of class $C^{1,\theta}$ and is isometric. Furthermore, by choosing a sequence of such $A_0^{k}\to \infty$ we can construct a sequence of isometric embeddings of class $C^{1,\theta}$ which uniformly approximate $u$. 

This concludes the proof of Theorem \ref{th-n-2n}.

\section{Proof of Proposition \ref{pr-stage}}\label{se-stage}

The proof will be completed by one observation on Kuiper's corrugations and absorbing certain higher order error terms in the metric decomposition. In this section all norms are taken on some bounded region $\Omega\subset\R^n.$

\subsection{A single step of perturbation}
To perturb the embedding $u:\Omega\to\R^{m}$ for iteration, we shall add primitive metrics in a single chart, i.e. in an open bounded subset $\Omega\subset\R^n$.  To keep the value of the given adapted short embedding in the isometric region, we add compactly supported primitive metrics via mollifying the metric and the given embedding at different length scales as in \cite{CaoSze2022}.  The difference here is that a single ``step" is for $n$ primitive metrics rather than only one primitive metric in \cite{Nash:1954vt, CoDeSz2012, DISz2018, CaoSze2022}. Such a ``step" of perturbation proceeds through the bound of the error resulting from the first Kuiper's corrugation and $n$ normal vectors. 

We start by recalling two lemmas.
The first one is about the Kuiper corrugation functions used in \cite{Kuiper:1955th,CoDeSz2012, HuWa2017, DISz2018}.
\begin{lemma}\label{le-gamma}
There exists $\epsilon>0$ and a smooth vector function $\Gamma=(\Gamma_1,\Gamma_2)(s, t)$ defined on $[0, \epsilon]\times\R$, which is a $2\pi$-periodic function in $t$, satisfying $(1+\partial_t\Gamma_1)^2+(\partial_t\Gamma_2)^2=1+s^2$ and the following estimates:
\begin{align}
\|\partial_t^j\Gamma_1(s, \cdot)\|_0\leq C(j)s^2,~~&
\|\partial_t^j\Gamma_2(s, \cdot)\|_0\leq C(j)s; \label{eq-gamma-c0}\\
\|\partial_s\partial_t^j\Gamma_1(s, \cdot)\|_0\leq C(j)s,~~&
\|\partial_s\partial_t^j\Gamma_2(s, \cdot)\|_0\leq C(j);\label{eq-gamma-c1}\\
\|\partial_s^2\partial_t^j\Gamma_1(s, \cdot)\|_0\leq C(j), ~~&
\|\partial_s^2\partial_t^j\Gamma_2(s, \cdot)\|_0\leq C(j),
\label{eq-gamma-c2}
\end{align}
for any $j\in\mathbb{N}.$
\end{lemma}
The second one is about the existence of normal vectors with respect to the given embedding, which is from \cite{CaoIn2020}.
\begin{lemma}\label{le-normal}
Let $\Omega\subset\R^n$ be a domain and $N\in\N$. Assume $u\in C^{N+1}(\bar\Omega,\R^{m})$ is an embedding such that
\begin{equation} \label{eq-normal}
\gamma^{-1}\mathrm{Id}\leq\nabla u^T\nabla u\leq \gamma \mathrm{Id}
\end{equation}
for some $\gamma>1.$ Then $u(\Omega)$ admits a family of normal vectors  $\{\zeta_1,\cdots,\zeta_k,\cdots, \zeta_{m-n}\}\subset C^{N}(\bar\Omega, \R^m)$ satisfying
\begin{align}
\zeta_i^T\zeta_k=\delta_{ik},\quad\nabla u^T\zeta_k=0,\quad [\zeta_i]_{j}\leq C(j,\gamma)(1+[v]_{j+1}),\label{eq-normalestimates}
\end{align}
for all $0\leq j\leq N.$ Here $\delta_{ik}=1$ when $i=k$ and vanishes else.
\end{lemma}

Now we are ready to proceed with a single step of the perturbation.

\begin{proposition}\label{pr-step} {\bf[Step]}
Let $\Omega\subset\R^n$ be a domain and $m>n$ be an integer. Let $u\in C^{2}(\Omega, \R^{m})$ be an embedding such that
\begin{equation}\label{eq-step-u}
{\gamma}^{-1}\mathrm{Id}\leq\nabla u^T\nabla u\leq\gamma \mathrm{Id},\textrm{ in }\Omega,
\quad  \|u\|_{2}\leq M\delta^{1/2}\nu,
\end{equation}
and $a_k, \Phi_k\in C^{2}(\Omega)$ satisfy
\begin{align}
&\|a_k\|_0\leq M\delta^{1/2},\quad \|a_k\|_1\leq M\delta^{1/2}\nu, \quad \|a_k\|_2\leq M\delta^{1/2}\nu\tilde\nu, \label{eq-step-a}\\
&\frac{1}{M}\leq |\nabla\Phi_k|\leq M \textrm{ in }\Omega, \quad
\|\nabla\Phi_k\|_1\leq M\nu, \quad \|\nabla\Phi_k\|_2\leq M\nu\tilde\nu \label{eq-step-phi}
\end{align}
for any $k=1,2,\cdots, m-n$  and some $M, \gamma\geq 1$, $\delta\leq 1$ and $\nu\leq\tilde\nu$. There exists a constant $c_0=c_0(M, \gamma)$ such that, for any $$\mu\geq c_0\tilde\nu,\quad \delta\leq c_0^{-1},$$
there exists a new embedding $v\in C^{2}(\Omega, \R^{m})$ such that
\begin{align}
{\bar\gamma}^{-1}\mathrm{Id}\leq\nabla v^T\nabla v&\leq \bar\gamma \mathrm{Id}\quad\textrm{ in }\Omega, \label{eq-step-0}\\
v&=u\textrm{ on }\Omega\setminus\big(\cup_{k=1}^{m-n}\supp a_k\big),\label{eq-step-1}\\
\|v-u\|_j&\leq C(\gamma)\delta^{1/2}\mu^{j-1},\, j=0, 1, \label{eq-step-2}\\
\|v\|_2&\leq C(\gamma)\delta^{1/2}\mu,\label{eq-step-3}\\
\left\|\nabla v^T\nabla v-\left(\nabla u^T\nabla u+\sum_{k=1}^{m-n}a_k^2\nabla\Phi_k\otimes\nabla\Phi_k\right)\right\|_j&\leq \overline M(\delta\nu\mu^{j-1}+\delta^2\mu^j), \,j=0, 1.\label{eq-step-4}
\end{align}
Here $\bar\gamma$ and $\overline M$ depend only on $M,\gamma$.
\end{proposition}

\begin{proof}
Initially taking $c_0\geq 1$ we have $\nu\leq\tilde\nu<\mu.$ We regularize $u$ on length-scale $\mu^{-1}$ to achieve smooth embedding $\tilde{u}$ satisfying
\begin{equation}\label{eq-tildeu-C2}
\|\tilde{u}-u\|_1\leq C(M)\delta^{1/2}\nu\mu^{-1} ,\quad \|\tilde{u}\|_{2}\leq C(M)\delta^{1/2}\nu,\quad \|\tilde{u}\|_{3}\leq C(M)\delta^{1/2}\nu\mu.
\end{equation}
Then by Proposition \ref{pr-mollification}, we immediately have $\|\nabla u-\nabla\tilde u\|_0\leq C(M)\delta^{1/2}\nu\mu^{-1}$ and then
\begin{equation}\label{eq-tildeu-C1}
{(2\gamma)}^{-1}\mathrm{Id}\leq \nabla\tilde u^T\nabla \tilde u\leq 2\gamma \mathrm{Id},
\end{equation}
provided that $c_0$ is large enough such that $c_0\geq 2\gamma C(M)$. Hence $\nabla\tilde u^T\nabla \tilde u$ is invertible and $\tilde u(\Omega)$ admits $m-n$ mutually orthogonal normal vectors $\{\tilde\zeta_k\}_{k=1}^{m-n}$ from Lemma \ref{le-normal}. Define
\begin{align*}
&\tilde\xi_k=\nabla \tilde{u}(\nabla \tilde{u}^T\nabla \tilde{u})^{-1}\nabla\Phi_k,\quad \xi_k=\frac{\tilde\xi_k}{|\tilde\xi_k|^2},\quad  \tilde{a}_k=|\tilde\xi_k|a_k,\quad \zeta_k=\frac{\tilde\zeta_k}{|\tilde\zeta_k||\tilde\xi_k|},
\end{align*}
Directly from construction, we have
\begin{equation}\label{eq-xizeta}
\nabla \tilde u^T\xi_k=\frac{\nabla\Phi_k}{|\tilde\xi_k|^2},\quad \nabla\tilde u^T\zeta_k=\nabla\tilde u^T\tilde\zeta_k=0, \quad \xi_k^T\zeta_{l}=0~(k, l=1,\cdots, m-n).
\end{equation}
It follows from \eqref{eq-tildeu-C2}-\eqref{eq-tildeu-C1} and \eqref{eq-step-u} that
\begin{equation}\label{eq-xizeta-est}
\begin{split}
\|(\xi_k, \zeta_k)\|_0&\leq C(\gamma), \quad \|(\xi_k,\zeta_k)\|_1\leq C(\gamma, M)\nu,\\
\|(\xi_k,\zeta_k)\|_2&\leq C(\gamma, M)\nu(\delta^{1/2}\mu+\tilde\nu)\leq C(\gamma, M)\nu\mu,
\end{split}
\end{equation}
and
\begin{equation}\label{eq-acm}
\begin{split}
\|\tilde a_k\|_0&\leq C(\gamma)\delta^{1/2},\quad
\|\tilde a_k\|_1
\leq C(\gamma, M)\delta^{1/2}\nu, \\
\|\tilde a_k\|_2&\leq C(\|a_k\|_2\|~|\tilde\xi_k|~\|_0+\|a_k\|_0\|~|\tilde\xi_k|~\|_2)
\leq C(\gamma, M)\delta^{1/2}\nu\mu,
\end{split}
\end{equation}
where the fact that $\tilde\nu\leq \mu$ is used in the last inequality. Define
$$
v=u+\frac{1}{\mu}\sum_{k=1}^{m-n}\bigl(\Gamma_1(\tilde a_k, \mu \Phi_k)\xi_k+\Gamma_2(\tilde a_k, \mu\Phi_k)\zeta_k\bigr).
$$
From the construction of $\Gamma$ in Lemma \ref{le-gamma}, \eqref{eq-step-1} follows.
Denoting 
$$
\Gamma_{ik}=\Gamma_i(\tilde a_k,\mu\Phi_k),\, \partial_t\Gamma_{ik}=\partial_t\Gamma_{i}(\tilde a_k,\mu\Phi_k), \, \partial_s\Gamma_{ik}=\partial_s\Gamma_{i}(\tilde a_k,\mu\Phi_k)
$$
for $i=1, 2$, by Proposition \ref{pr-composition} and \eqref{eq-interpolation} we have
\begin{equation}\label{eq-interu}
\|v-u\|_{j}
\leq\frac{C}{\mu}\sum_{k=1}^n(\|\Gamma_{1k}\|_{j}\|\xi_k\|_{0}+\|\Gamma_{1k}\|_0\|\xi_k\|_{j}+
\|\Gamma_{2k}\|_{j}\|\zeta_k\|_0+\|\Gamma_{2k}\|_0\|\zeta_k\|_{j})
\end{equation}
for $j=0, 1, 2$, 
where $\|\Gamma_{ik}\|_{j}$ denote the $C^j$-norms in $x\in\Omega$ of the composition 
$$
x\mapsto \Gamma_i(\tilde{a}_k(x), \mu\Phi_k(x)).
$$
Using Lemma \ref{le-gamma} and the assumptions \eqref{eq-step-a}-\eqref{eq-step-phi}, applying Proposition \ref{pr-composition}, we deduce that
\begin{equation}\label{eq-Gammaest1}
\begin{split}
&\|\Gamma_{1k}\|_0+\|\partial_t\Gamma_{1k}\|_0+\|\partial_t^2\Gamma_{1k}\|_0\leq C\|\tilde a_k^2\|_0\leq C(\gamma)\delta,\\
&\|\Gamma_{1k}\|_1\leq\|\partial_t\Gamma_{1k}\|_0\|\nabla\Phi_k\|_0\mu+\|\partial_s\Gamma_{1k}\|_0\|\nabla\tilde a_k\|_0\\
&\quad\quad\quad \leq C(\gamma)\delta\mu+C(M,\gamma)\delta\nu\leq C(\gamma)\delta\mu,\\
&\|\partial_t\Gamma_{1k}\|_1\leq\|\partial_t^2\Gamma_{1k}\|_0\|\nabla\Phi_k\|_0\mu
+\|\partial_s\partial_t\Gamma_{1k}\|_0\|\nabla\tilde a_k\|_0
\leq C(\gamma)\delta\mu,
\end{split}
\end{equation}
and
\begin{equation}\label{eq-Gammaest2}
\begin{split}
&\|\Gamma_{2k}\|_0+\|\partial_t\Gamma_{2k}\|_0+\|\partial_t^2\Gamma_{2k}\|_0\leq C\|\tilde a_k\|_0\leq C(\gamma)\delta^{1/2},\\
&\|\Gamma_{2k}\|_1\leq \|\partial_t\Gamma_{2k}\|_0\|\nabla\Phi_k\|_0\mu+\|\partial_s\Gamma_{2k}\|_0\|\nabla\tilde a_k\|_0
\leq C(\gamma)\delta^{1/2}\mu,\\
&\|\partial_t\Gamma_{2k}\|_1\leq \|\partial_t^2\Gamma_{2k}\|_0\|\nabla\Phi_k\|_0\mu+\|\partial_s\partial_t\Gamma_{2k}\|_0\|\nabla\tilde a_k\|_0
\leq C(\gamma)\delta^{1/2}\mu,
\end{split}
\end{equation}
where we have chosen $c_0=c_0(M,\gamma)$ such that $\mu\geq C(M, \gamma)\nu$. Similarly, we also have
\begin{equation}\label{eq-Gammaest12}
\begin{split}
\|\partial_s\Gamma_{1k}\|_0&\leq C(\gamma)\|\tilde a_k\|_0\leq C(\gamma)\delta^{1/2},\quad
\|\partial_s\Gamma_{2k}\|_0\leq C(\gamma),\\
\|\partial_s\Gamma_{1k}\|_1&\leq \|\partial_t\partial_s\Gamma_{1k}\|_0\|\nabla\Phi_k\|_0\mu+\|\partial_s^2\Gamma_{1k}\|_0\|\nabla\tilde a_k\|_0\leq C(\gamma)\delta^{1/2}\mu,\\
\|\partial_s\Gamma_{2k}\|_1&\leq \|\partial_t\partial_s\Gamma_{2k}\|_0\|\nabla\Phi_k\|_0\mu+\|\partial_s^2\Gamma_{2k}\|_0\|\nabla\tilde a_k\|_0\leq C(\gamma)\mu.
\end{split}
\end{equation}
Thus by \eqref{eq-Gammaest1}-\eqref{eq-Gammaest12}, we derive
\begin{align*}
\|v-u\|_0&\leq C(\gamma)\delta^{1/2}\mu^{-1},\\
\|v-u\|_1&\leq C(\gamma)\delta^{1/2}+C(\gamma, M)\delta^{1/2}\nu\mu^{-1}\leq C(\gamma)\delta^{1/2},\\
\|v-u\|_2&\leq C(\gamma)\delta^{1/2}\mu+C(\gamma, M)\delta^{1/2}\nu\leq C(\gamma)\delta^{1/2}\mu.
\end{align*}
In a nutshell, \eqref{eq-step-2} is achieved since $\mu\geq c_0(M,\gamma)\nu\geq M\nu$ and \eqref{eq-step-3} also follows.
\smallskip

To estimate the  metric difference, we take gradient of $v$ and get
\begin{align*}
\nabla v=&\nabla u+\sum_{k=1}^{m-n}(\partial_t\Gamma_{1k}\xi_k\otimes\nabla\Phi_k+\partial_t\Gamma_{2k}\zeta_k\otimes\nabla\Phi_k)+
\sum_{k=1}^{m-n}\frac{1}{\mu}(\Gamma_{1k}\nabla \xi_k+\Gamma_{2k}\nabla \zeta_k)\\
&+\sum_{k=1}^{m-n}\frac{1}{\mu}(\partial_s\Gamma_{1k}\xi_k\otimes\nabla\tilde a_k+\partial_s\Gamma_{2k}\zeta_k\otimes\nabla\tilde a_k)\\
=&\nabla u+\sum_{k=1}^{m-n}(B_k+E_{1k}+E_{2k}),
\end{align*}
where based on the normal directions and tangential direction we have set $B_k=B_k^{(1)}+B_k^{(2)}$, $E_{1k}=E_{1k}^{(1)}+E_{1k}^{(2)}$ and $E_{2k}=E_{2k}^{(1)}+E_{2k}^{(2)}$ with
\begin{align*}
&B_k^{(1)}=\partial_t\Gamma_{1k} \xi_k\otimes\nabla\Phi_k,\quad B_k^{(2)}=\partial_t\Gamma_{2k}\zeta_k\otimes\nabla\Phi_k,\\
&E_{1k}^{(1)}=\frac{1}{\mu}\left(\Gamma_{1k}\nabla \xi_k+\frac{\Gamma_{2k}\nabla\tilde\zeta_k}{|\tilde\zeta_k||\tilde\xi_k|}-
\frac{\Gamma_{2k}\tilde\zeta_k\otimes\nabla|\tilde\zeta_k|}{|\tilde\zeta_k|^2|\tilde\xi_k|}\right),\quad
E_{1k}^{(2)}=-\frac{\Gamma_{2k}}{\mu}\frac{\tilde\zeta_k\otimes\nabla|\tilde\xi_k|}{|\tilde\zeta_k||\tilde\xi_k|^2},
\end{align*}
where we have calculated that
$$
\nabla \zeta_k=\frac{\nabla\tilde\zeta_k}{|\tilde\zeta_k||\tilde\xi_k|}
-\frac{\tilde\zeta_k\otimes\nabla|\tilde\zeta_k|}{|\tilde\zeta_k|^2|\tilde\xi_k|}
-\frac{\tilde\zeta_k\otimes\nabla|\tilde\xi_k|}{|\tilde\zeta_k||\tilde\xi_k|^2},
$$
and
$$
E_{2k}^{(1)}=\frac{1}{\mu}\partial_s\Gamma_{1k}\xi_k\otimes\nabla\tilde a_k,\quad
E_{2k}^{(2)}=\frac{1}{\mu}\partial_s\Gamma_{2k}\zeta_k\otimes\nabla\tilde a_k.
$$
With the notation $\textrm{sym}(H)=(H+H^T)/2$, the metric induced by $v$ can be written as
\begin{align*}
\nabla  v^T\nabla v=&\nabla u^T\nabla u+\sum_{k=1}^{m-n}2 \textrm{sym}(\nabla u)^TB_k+\sum_{k=1}^{m-n} B_k^TB_k\\
&+\sum_{k=1}^{m-n}2 \textrm{sym}(\nabla u+B_k)^T(E_{1k}+E_{2k})+\sum_{k=1}^{m-n}(E_{1k}+E_{2k})^T(E_{1k}+E_{2k})\\
&+\sum_{i,k=1, i\neq k}^{m-n}2\textrm{sym}(B_k+E_{1k}+E_{2k})^T(B_i+E_{1i}+E_{2i})
\end{align*}
By \eqref{eq-xizeta} and Lemma \ref{le-gamma}, we have
$$	
\nabla \tilde{u}^TB_k+B_k^T\nabla\tilde{u}+B_k^TB_k=a_k^2\nabla\Phi_k\otimes\nabla\Phi_k.
$$
By orthogonality, it holds that
\begin{equation}\label{eq-bkek-orth}
\nabla\tilde u^T(E_{1k}^{(2)}+E_{2k}^{(2)})=0,\quad (B_k^{(1)})^TB_i^{(2)}=(B_k^{(2)})^TB_i^{(2)}=0, \,\,i,k=1,\cdots, m-n,
\end{equation}
which play vital roles in later calculation. Hence we further have
\begin{equation}\label{eq-metricerror1}
\begin{split}
&\nabla  v^T\nabla v-\left(\nabla u^T\nabla u+\sum_{k=1}^{m-n} a_k^2\nabla\Phi_k\otimes\nabla\Phi_k\right)\\
=&\sum_{k=1}^{m-n}2\textrm{sym}[(\nabla u-\nabla\tilde u)^TB_k+(\nabla u-\nabla\tilde u+B_k)^T(E_{1k}^{(2)}+E_{2k}^{(2)})]\\
&+\sum_{k=1}^{m-n}2\textrm{sym} (\nabla u+B_k)^T(E_{1k}^{(1)}+E_{2k}^{(1)})+\sum_{k=1}^{m-n}(E_{1k}+E_{2k})^T(E_{1k}+E_{2k})\\
&+\sum_{i,k=1, i\neq k}^{m-n}2\textrm{sym}[(B_k^{(1)})^TB_i^{(1)}+B_k^T(E_{1i}+E_{2i})+(E_{1k}+E_{2k})^T(E_{1i}+E_{2i})].
\end{split}
\end{equation}
Comparing with the calculations in \cite[Proposition 3.1]{CaoSze2022}, the only difference is the existence of nonlinear interaction terms from tangential direction $(B_k^{(1)})^TB_i^{(1)} (i\neq k)$ in the last line of \eqref{eq-metricerror1}. However, with the special structure of the first corrugation, from the estimates \eqref{eq-xizeta-est}, \eqref{eq-acm}, \eqref{eq-Gammaest1}, \eqref{eq-Gammaest2} and \eqref{eq-Gammaest12},  we have the bounds
\begin{equation}\label{eq-bk1}
	\|B_k^{(1)}\|_0\leq C(\gamma)\delta, \quad \|B_k^{(2)}\|_0\leq C(\gamma)\delta^{1/2},
\end{equation}
which then gives
\begin{equation}\label{eq-bki1}
	(B_k^{(1)})^TB_i^{(1)}\leq C(\gamma)\delta^2.
\end{equation}
Furthermore, we obtain
\begin{equation}\label{eq-ek12}
	\begin{split}
\|E_{1k}^{(1)}\|_0&\leq C(\gamma, M)\mu^{-1}\delta\nu,\quad \|E_{2k}^{(1)}\|_0\leq C(\gamma, M)\mu^{-1}\delta\nu,\\
\|E_{1k}^{(2)}\|_0&+\|E_{2k}^{(2)}\|_0\leq C(\gamma, M)\mu^{-1}\delta^{1/2}\nu.
\end{split}
\end{equation}
Using $\mu\geq c_0\nu$,  from \eqref{eq-bki1} and \eqref{eq-ek12}, we finally deduce 
\begin{equation}\label{eq-metricerror0}
\left\|\nabla v^T\nabla v-\left(\nabla u^T\nabla u+\sum_{k=1}^{m-n}a_k^2\nabla\Phi_k\otimes\nabla\Phi_k\right)\right\|_0\leq C(\gamma, M)(\delta\mu^{-1}\nu+\delta^2).
\end{equation}
Similarly, using the Leibniz-rule we obtain
\begin{align*}
\|B_k^{(1)}\|_1&\leq C(\gamma, M)\delta\mu, \quad \|B_k^{(2)}\|_1\leq C(\gamma, M)\delta^{1/2}\mu,\\
\|E_{1k}^{(1)}\|_1&\leq C(\gamma, M)\delta\nu,\quad \|E_{2k}^{(1)}\|_1\leq C(\gamma, M)\delta\nu,\\
\|E_{1k}^{(2)}\|_1&+\|E_{2k}^{(2)}\|_1\leq C(\gamma, M)\delta^{1/2}\nu.
\end{align*}
Then it follows that
\begin{equation*}
\left\|\nabla v^T\nabla v-\left(\nabla u^T\nabla u+\sum_{k=1}^{m-n}a_k^2\nabla\Phi_k\otimes\nabla\Phi_k\right)\right\|_{1}\leq C(\gamma, M)(\delta\nu+\delta^2\mu),
\end{equation*}
which together with \eqref{eq-metricerror0} implies \eqref{eq-step-4}.

Finally, we shall verify that $v$ is a bounded embedding. It follows from \eqref{eq-metricerror0},  that
\begin{align*}
\left\|\nabla v^T\nabla v-\left(\nabla u^T\nabla u+\sum_{k=1}^{m-n}a_k^2\nabla\Phi_k\otimes\nabla\Phi_k\right)\right\|_0\leq \frac{1}{2\gamma},
\end{align*}
provided taking $c_0\geq 2\gamma C(M, \gamma)$ and $\delta< c_0^{-1}$. Using \eqref{eq-step-phi} and $\delta\leq 1$ we get
$$
0\leq \sum_{k=1}^{m-n}a_k^2\nabla\Phi_k\otimes\nabla\Phi_k\leq (m-n)M^2\mathrm{Id}.
$$
 By \eqref{eq-step-u}, we deduce \eqref{eq-step-0}, which also implies that $v$ is an immersion. The argument for that $v$ is an embedding is standard and we can also find a proof in \cite{DISz2018}. The proof is then completed.
\end{proof}

\bigskip

With a single step of perturbation established in Proposition \ref{pr-step}, we can prove Proposition \ref{pr-stage} through several steps. Due to the decomposition of metrics for different $n,$ the proof is divided into the following two subsections.

\subsection{Proof for the case $n=2$}

The goal is to show that when all the conditions in Proposition \ref{pr-stage} hold for $n=2$ we can conclude \eqref{eq-stage-1} and
\begin{align}
\|v-u\|_0&\leq C\delta^{1/2}\lambda^{-\kappa}\leq C\delta^{1/2}\lambda^{-(\kappa+1)/2},\quad
\|v-u\|_1\leq C\delta^{1/2},
\quad \|v\|_2\leq C\delta^{1/2}\lambda^\kappa,\label{eq-stage-v}\\
&\|\mathcal{E}\|_0\leq C\delta\lambda^{1-\kappa}+C\delta^2, \quad
\|\mathcal{E}\|_1\leq C\delta\lambda+C\delta^2\lambda^\kappa.\label{eq-stage-error}
\end{align}
Here $C\geq 1$ is the constant in Proposition \ref{pr-stage}.

With the help of Lemma \ref{le-conformal} and Proposition \ref{pr-step}, we are then able to add the term $\rho^2(G+H)$ through a single step.

\begin{proof}[Proof of Proposition \ref{pr-stage} when $n=2$] The proof proceeds entirely analogously to that of \cite[Corollary 3.1]{CaoSze2022}. We provide the proof here in detail for the convenience of the reader.
 We first obtains $\tilde \rho$, $\tilde G$ and $\tilde H$ by mollifying $\rho$, $G$ and $H$ respectively at length-scale $\ell=\lambda^{-\kappa}$. By \eqref{eq-stage-rho}-\eqref{eq-stage-H} and Proposition \ref{pr-mollification}, we have
\begin{equation}\label{eq-moli-rhogh}
\begin{split}
\|\tilde \rho\|_0\leq \delta^{1/2}, \quad &\|\tilde G\|_0\leq C, \quad \|\tilde H\|_{0}\leq \lambda^{-\alpha},\quad\\
\|\tilde \rho\|_j\leq C_j\delta^{1/2}\lambda\ell^{1-j},\quad &\|\tilde G\|_{j}\leq C_j(\gamma)\ell^{1-j},\quad  \|\tilde H\|_{j}\leq C_j\lambda^{1-\alpha}\ell^{1-j},\\
\|\tilde \rho-\rho\|_0\leq C\delta^{1/2}\lambda\ell,\quad&\|\tilde G-G\|_0\leq C(\gamma)\ell,\quad \|\tilde H-H\|_0\leq C\lambda^{1-\alpha}\ell\,,
\end{split}
\end{equation}
for all $j\geq 1$. Moreover, for any $0<\alpha'\leq \alpha<1$ it holds that
\begin{align*}	
\|\tilde G+\tilde H\|_{\alpha'}
\leq C\gamma.
\end{align*}
Note that ${\gamma}^{-1}\mathrm{Id}\leq \tilde G\leq \gamma\mathrm{Id}$, which then implies that
$$
\tilde G+\tilde H\geq ({\gamma}^{-1}-\lambda^{-\alpha})\mathrm{Id}\geq ({2\gamma})^{-1}\mathrm{Id},
$$
provided that $\lambda^{\alpha}\geq 2\gamma$, which is ensured by taking $\lambda_*$ in \eqref{eq-stage-lambda0} large. Lemma \ref{le-conformal} can be used to get $\tilde a, \Phi=(\Phi_1, \Phi_2)$ such that
$$
\tilde G+\tilde H=\tilde a^2(\nabla\Phi_1\otimes\nabla\Phi_1+\nabla\Phi_2\otimes\nabla\Phi_2),
$$
$\tilde{a}\geq C(\gamma), \,\det(D\Phi)\geq C(\gamma),$ and
\begin{align*}
\|\tilde{a}\|_{j+\alpha'}+\|\nabla\Phi\|_{j+\alpha'}\leq C_j(\alpha',\gamma)\lambda^{1-\alpha}\ell^{1-j-\alpha'}.
\end{align*}
Taking $\alpha'\leq \frac{\alpha}{\kappa}$ and using $\ell=\lambda^{-\kappa},$ we directly have
\begin{equation}\label{eq-phi-estimate}
\begin{split}
&{C(\gamma)}^{-1}\leq|\nabla\Phi_1|,\,|\nabla\Phi_2|\leq C(\gamma);\\
&\|\nabla\Phi\|_{j}\leq C_j(\gamma)\lambda\ell^{1-j}\,,\quad
\|\tilde a\|_{j}\leq C_j(\gamma)\lambda\ell^{1-j}\,
\end{split}
\end{equation}
for any $j\geq 1$. Let $\mathbf{h}=\rho^2(G+H),\,\tilde{\mathbf{h}}=\tilde \rho^2(\tilde G+\tilde H)$ and
$a=\tilde a \tilde \rho$, then
\begin{equation}\label{eq-decompstage}
\tilde{\mathbf{h}}=a^2(\nabla\Phi_1\otimes\nabla\Phi_1+\nabla\Phi_2\otimes\nabla\Phi_2)
\end{equation}
and $a$ satisfies
\begin{equation}\label{eq-vartheta}
\|a\|_0\leq C(\gamma)\delta^{1/2},\quad
\|a\|_1\leq C(\gamma)\delta^{1/2}\lambda,\quad
\|a\|_2\leq  C(\gamma)\delta^{1/2}\lambda\ell^{-1}\,.
\end{equation}
Take $\delta_*=c_0^{-1}$ as in Proposition \ref{pr-step}. With \eqref{eq-phi-estimate}-\eqref{eq-vartheta}, for any $\delta<\delta_*,$ Proposition \ref{pr-step} can be applied with
$$
m=4, \, n=2,\quad \nu=\lambda, \quad \tilde\nu=\ell^{-1},\quad \mu=\lambda^\kappa\,,
$$
to yield a new embedding $v\in C^2(\Omega;\R^{4})$ satisfying
\begin{align*}
&v=u\textrm{ on }\Omega\setminus \supp \tilde{\mathbf{h}}=\Omega\setminus(\supp\rho+\B_{\lambda^{-\kappa}}(0)),\\
\|v-u\|_0\leq &C(\gamma)\delta^{1/2}\lambda^{-\kappa},\quad
\|v-u\|_1\leq C(\gamma)\delta^{1/2},\quad
\|v\|_2\leq C(\gamma)\delta^{1/2}\lambda^{\kappa},
\end{align*}
which implies \eqref{eq-stage-1} and \eqref{eq-stage-v}. Moreover, the new metric error
\begin{align*}
\mathcal{E}&=\nabla v^T\nabla v-(\nabla u^T\nabla u+\rho^2(G+H))
\end{align*}
satisfies
\begin{align*}
\|\mathcal{E}\|_0&\leq C\delta \lambda^{1-\kappa}+C\delta^2+\|\mathbf{h}-\tilde{\mathbf{h}}\|_0\\
\|\mathcal{E}\|_1&\leq C\delta\lambda+C\delta^2\lambda^\kappa+\|\mathbf{h}-\tilde{\mathbf{h}}\|_1\,.
\end{align*}
By \eqref{eq-moli-rhogh}, it is not hard to get
\begin{equation}\label{eq-h-molify}
\begin{split}	
\|\tilde{\mathbf{h}}-\mathbf{h}\|_0\leq C\delta\lambda^{1-\kappa},\quad
\|\tilde {\mathbf{h}}-\mathbf{h}\|_1\leq C\delta\lambda\,.
\end{split}
\end{equation}
This concludes \eqref{eq-stage-error}.
\end{proof}

\subsection{Proof for the case $n\geq3$}
In this case, we shall derive \eqref{eq-stage-1} and
\begin{align}
&\|v-u\|_0\leq C\delta^{1/2}\lambda^{-(\kappa+1)/2}, \quad
\|v-u\|_1\leq C\delta^{1/2}, \quad
\|v\|_2\leq C\delta^{1/2}\lambda^{\frac{n+1}{2}(\kappa-1)+1}, \label{eq-stage-v-1}\\
&\|\mathcal{E}\|_0\leq C(\delta\lambda^{1-\kappa}+\delta^2), \quad
\|\mathcal{E}\|_1\leq C(\delta\lambda^{\frac{n-1}{2}(\kappa-1)+1}+\delta^2\lambda^{\frac{n+1}{2}(\kappa-1)+1}).\label{eq-stage-error-1}
\end{align}
Here $C\geq 1$ is the constant in Proposition \ref{pr-stage} and we have used the definition of $\theta(n).$

For $n\geq3$, the term  $\rho^2(G+H)$ can be decomposed using lemma \ref{le-decompose}.
Thus we can decompose the term $\rho^2(G+H)$ into $n_*=n\cdot(n+1)/2$ primitive metrics for $n\geq3.$
We shall note two facts: only $n$ normal vectors exist since the codimensions are only $2n-n=n$, and the factor $(n+1)/2$ may be not an integer. Hence, the later proof is divide into two subcases.

\subsubsection{Proof of Proposition \ref{pr-stage} for odd $n\geq3$} In this case, the factor $(n+1)/2$ is an integer and we can complete the proof through $(n+1)/2$ times application of Proposition \ref{pr-step}.

\begin{proof}
First mollify the term $\mathbf{h}=\rho^2(G+H)$ to get $\tilde{\mathbf{h}}=\tilde \rho^2(\tilde G+\tilde H)$ as in the case $n=2$ at the length scale $\ell=\lambda^{-\kappa}$. Then we have
$$
C(\gamma)\mathrm{Id}\geq\tilde G+\tilde H\geq (2\gamma)^{-1}\mathrm{Id}.
$$
Lemma \ref{le-decompose} implies
$$
\tilde G+\tilde H=\sum_{k=1}^{n_*}\tilde a_k^2\xi_k\otimes\xi_k,
$$
where each $\tilde a_k$ satisfy
\begin{equation*}
C(\gamma)^{-1}\leq\tilde{a}_k\leq C(\gamma), \quad \|\tilde a_k\|_{j}\leq C_j(\gamma)\lambda\ell^{1-j}\,
\end{equation*}
for any $j\geq 1$. Let $a_k=\tilde a_k \tilde \rho$, then
\begin{equation}\label{eq-decomp-nd}
\tilde{\mathbf{h}}=\sum_{k=1}^{n_*}a_k^2\xi_k\otimes\xi_k
\end{equation}
and $a_k=\sqrt{L_k(\tilde{\mathbf{h}})}$ satisfy
\begin{equation}\label{eq-ak}
\|a_k\|_0\leq C(\gamma)\delta^{1/2},\quad\|a_k\|_1\leq C(\gamma)\delta^{1/2}\lambda,\quad
\|a_k\|_2\leq C(\gamma)\delta^{1/2}\lambda\ell^{-1}\,.
\end{equation}

Set $u_0=u$, $\gamma_0=\gamma$ and $\delta_*=c_0^{-1}$ as in Proposition \ref{pr-step}. From \eqref{eq-ak} and \eqref{eq-stage-u}, all the conditions of Proposition \ref{pr-step} are satisfied for $u_0, \gamma_0$ and $m=2n,$  with $\delta_0=\delta,\,\nu_0=\lambda, \, \tilde\nu_0=\lambda^\kappa,\, \Phi_k=x\cdot\xi_k.$
Successive application of Proposition \ref{pr-step} to obtain $\{u_l\}_{l=1,\cdots,\frac{n+1}{2}}$ with constants
\begin{align*}
\delta_l=\delta,\quad\nu_{l}=\tilde\nu_l=\lambda_{l},\quad \lambda_l=\lambda K^{l}, \,K=\lambda^{\kappa-1}, \quad \gamma_l=\bar\gamma_{l-1},
\end{align*}
and for $l=1,\cdots, \frac{n+1}{2},$
\begin{equation}\label{eq-stepjth}
\begin{split}
&\|u_{l}-u_{l-1}\|_1\leq C(\gamma)\delta^{1/2}\lambda_l^{-1}\\
&\|u_{l}-u_{l-1}\|_1\leq C(\gamma)\delta^{1/2},\quad \|u_{l}\|_2\leq C(\gamma)\delta^{1/2}\lambda_l,\\
&\left\|\nabla u_{l}^T\nabla u_{l}-\left(\nabla u_{l-1}^T\nabla u_{l-1}+ \sum_{k=n(l-1)+1}^{nl}a_k^2\xi_k\otimes\xi_k\right)
\right\|_j\\
&\quad \leq C(\gamma_l)(\delta\lambda_{l-1}\lambda_l^{j-1}+\delta^2\lambda_l^j),\quad j=0, 1.
\end{split}
\end{equation}
Then $v=u_{\frac{n+1}{2}}\in C^2(\Omega,\R^{2n})$ is the desired embedding satisfying
\begin{align*}
\|v-u\|_0\leq&\sum_{l=1}^{\frac{n+1}{2}}\|u_l-u_{l-1}\|_0\leq C(\gamma)\delta^{1/2}\lambda^{-\kappa},\\
\|v-u\|_1\leq&\sum_{l=1}^{\frac{n+1}{2}}\|u_l-u_{l-1}\|_1\leq C(\gamma)\delta^{1/2},\\
\|v\|_{2}\leq& C(\gamma)\delta^{1/2}\lambda K^{\frac{n+1}{2}},
\end{align*}
from which \eqref{eq-stage-v-1} follows. Set the new metric error as
\begin{align*}
\mathcal{E}&=\nabla v^T\nabla v-(\nabla u^T\nabla u+\rho(G+H))\\
&=\mathcal{E}_0+\tilde{\mathbf{h}}-\mathbf{h},
\end{align*}
where
\begin{align*}
\mathcal{E}_0&=\nabla v^T\nabla v-\left(\nabla u^T\nabla u+\sum_{k=1}^{n_*}a_k^2\xi_k\otimes\xi_k\right)\\
&=\sum_{l=1}^{\frac{n+1}{2}}\left(\nabla u_l^T\nabla u_l-\left(\nabla u_{l-1}^T\nabla u_{l-1}+\sum_{k=n(l-1)+1}^{nl}a_l^2\xi_l\otimes\xi_l\right)\right).
\end{align*}
\eqref{eq-stage-v-1} can be obtained as in the case $n=2$.
By \eqref{eq-stepjth} we obtain
\begin{align*}
\|\mathcal{E}_0\|_0&\leq \sum_{l=1}^{\frac{n+1}2}C(\gamma)(\delta\lambda_{l-1}\lambda_l^{-1}+\delta^2)
\leq C(\gamma)(\delta\lambda^{1-\kappa}+\delta^2),\\
\|\mathcal{E}_0\|_1&\leq \sum_{l=1}^{\frac {n+1}2}C(\gamma)(\delta\lambda_{l-1}+\delta^2\lambda_{l})
\leq C(\gamma)(\delta\lambda^{\frac{n-1}{2}(\kappa-1)+1}+\delta^2\lambda^{\frac{n+1}{2}(\kappa-1)+1}).
\end{align*}
which together with \eqref{eq-h-molify} implies \eqref{eq-stage-error-1}.
\end{proof}
\smallskip
\subsubsection{Proof of Proposition \ref{pr-stage} for even $n\geq3$}\label{s-even3}

When $n$ is even, we shall perform $n/2+1$ steps in one stage. If we apply Proposition \ref{pr-step}, then the bound of the $\|v\|_2$ will be larger than that in Proposition \ref{pr-stage}, which is the key difficulty. To overcome this obstacle, we will make use of the strategy, originating from \cite{Kallen1978}, which is based on decomposing not just the metric error term $\rho^2(G+H)$ based on Lemma \ref{le-decompose}, but, as in \cite{DI2020,CaoIn2020}, absorb part of new error terms. To this end we use the following variant of Lemma \ref{le-decompose} from \cite{DI2020}:

\begin{lemma}\label{le-perturb}
Let $N_0\leq n_*$ be some integer and $P_0\in \R^{n\times n}_{sym,+}$ with $\mathsf{osc}_\Omega P_0\leq \frac14\sigma_0$. There exists a geometric constant $\sigma_1>0$ and vectors $\xi_1,\cdots,\xi_k,\cdots \xi_{n*} \in \mathbb S^{n-1}$ with the following property.
If  $P$, $\{\Lambda_k\}_{k=1}^{N_0},$ and $\{\Theta_{kl}\}_{k, l=1}^{N_0}\in C^1(\bar\Omega; \R^{n\times n}_{sym})$ fulfill
 \begin{equation}\label{eq-peturb}
 \|P-P_0\|_0+\sum_{k=1}^{N_0}\|\Lambda_k\|_0+\sum_{k,l=1}^{N_0}\|\Theta_{kl}\|_0\leq \sigma_1,
 \end{equation}
then there exist $C^1$ functions $a_1, \cdots, a_{n_*}:\bar\Omega\to\R$  given as
\begin{equation}\label{eq-a_i}
a_i(x) = \Psi_i(P(x), \{ \Lambda_k(x)\},\{\Theta_{kl}(x)\})\,
\end{equation}
satisfying
\begin{equation*}
P=\sum_{i=1}^{n_*}a_i^2\xi_i\otimes\xi_i+\sum_{k=1}^{N_0}a_k\Lambda_k+\sum_{k,l=1}^{N_0}a_ka_l\Theta_{kl}\,,
\end{equation*}
and
\begin{equation}\label{eq-aicj}
\|a_i\|_{j}\leq C_{j}\left(\|P\|_{j}+\sum_{k=1}^{N_0}\|\Lambda_k\|_j+\sum_{k,l=1}^{N_0}\|\Theta_{kl}\|_j\right),
\end{equation}
for $j=0, 1$ and $1\leq i\leq n_*.$ Here, the constants $C_{j}\geq 1$ depend only on $j, \sigma_0,\sigma_1$.
\end{lemma}
 \noindent Although Lemma \ref{le-perturb} is about the perturbation with indexes $k, l$ varying only from 1 to $N_0(\leq n_*)$,  its proof is same as Proposition 5.4 in \cite{DI2020}.

 However, codimensions are not enough to absorb all errors as in \cite{DI2020, CaoIn2020}. The key point is to apply such decomposition to absorb the error resulting from  the first step when adding the first $n/2$ primitive metrics. The remained ${n^2}/{2}$ primitive metrics are then added through $n/2$ times using Proposition \ref{pr-step}.

 \begin{proof} In the first step we apply the decomposition in Lemma \ref{le-perturb} and use Nash's spirals to add the first $n/2$ primitive metrics. 

We regularize $u$ at length scale $\lambda^{-\tau}$  to get $\tilde u\in C^{\infty}(\bar \Omega)$ with $\tau=(\kappa+1)/{2}>1$ after taking $\kappa>1$.
It then follows that
$$
(2\gamma)^{-1}\mathrm{Id}\leq\nabla\tilde u^T\nabla\tilde u\leq (2\gamma)\mathrm{Id},
$$
provided taking $\lambda_*$ large enough such that $\lambda^{1-\tau}\leq C(\gamma)^{-1}$ for some constant $C(\gamma)$. 
Thus by Lemma \ref{le-normal}, there exist
 $n$ unit normal vectors $\{\zeta_k, \eta_k, k=1,\cdots, n/2\}$ to the manifold $\tilde u(\bar\Omega)$ satisfying the estimates \eqref{eq-normal}. Fix the vectors $\xi_1, \cdots, \xi_{n_*}\in \mathbb S^{n-1}$ in Lemma \ref{le-perturb} for $\gamma\geq 1$ here. For $1\leq k\leq n/2$, define as in \cite{DI2020,CaoIn2020}
\begin{align*}
A_k=&\cos(\lambda^\tau\xi_k\cdot x)\zeta_k\otimes\xi_k-\sin(\lambda^\tau\xi_k\cdot x)\eta_k\otimes\xi_k,\\
B_k=&\sin(\lambda^\tau\xi_k\cdot x)\nabla\zeta_k+\cos(\lambda^\tau\xi_k\cdot x)\nabla\eta_k,\\
D_k=&\sin(\lambda^\tau\xi_k\cdot x)\zeta_k+\cos(\lambda^\tau\xi_k\cdot x)\eta_k\,,
\end{align*}
which satisfy
\begin{equation}\label{eq-abdk}
\begin{split}
&\|A_k\|_0+\|D_k\|_0\leq C(1+\|\nabla\tilde u\|_0)\leq C,\\
&\|A_k\|_1+\|D_k\|_1\leq C(\lambda^\tau\|\nabla\tilde u\|_0+\|\nabla^2\tilde u\|_0)\leq C\lambda^\tau,\\
&\|B_k\|_0
\leq C\delta^{1/2}\lambda,\qquad
\|B_k\|_1
\leq C\delta^{1/2}\lambda^{1+\tau}\,.
\end{split}
\end{equation}
Using $\nabla\tilde u^TA_k=0,$ we get
\begin{equation}\label{eq-uak}
\begin{split}
\|\nabla u^T A_k\|_0
\leq C\delta^{1/2}\lambda^{1-\tau},\quad
\|\nabla u^TA_k\|_1\leq
C\delta^{1/2}\lambda\,,
\end{split}
\end{equation}
and for $\nabla u^T D_k$ it also holds that
\begin{equation}\label{eq-udk}
\|\nabla u^T D_k\|_0 \leq C \delta^{1/2}\lambda^{1-\tau}\,, \quad \|\nabla u^{T}D_k\|_1 \leq C\delta^{1/2} \lambda\,.
\end{equation}
Set
 \begin{align*}
 \Lambda_k=&2\,\textrm{sym}(\nabla u^T A_k)+2\lambda^{-\tau}\textrm{sym}(\nabla u^TB_k)\,,\\
 \Theta_{kl}=&2\lambda^{-\tau}\textrm{sym}(A_k^TB_l)+2\lambda^{-2\tau}\textrm{sym}(B_k^TB_l).
 \end{align*}
Taking value of \eqref{eq-abdk}-\eqref{eq-udk}, we deduce
\begin{equation}\label{eq-LTkl}
\begin{split}
\|\Lambda_k\|_0 
\leq& C\delta^{1/2}\lambda^{1-\tau}\,, \quad \|\Theta_{kl}\|_0\leq C\delta^{1/2}\lambda^{1-\tau}\,,\\
\|\Lambda_k\|_1\leq& C\left (\|\nabla u^TA_k\|_1+\lambda^{-\tau}\left (\|\nabla u\|_1\|B_k\|_0+\|\nabla u\|_0\|B_k\|_1\right ) \right )\leq C\delta^{1/2}\lambda\,,\\
\|\Theta_{kl}\|_1\leq &C\lambda^{-\tau}(\|A_k\|_1\|B_l\|_0+\|A_k\|_0\|B_l\|_1+\lambda^{-\tau}(\|B_k\|_1\|B_l\|_0+\|B_k\|_0\|B_l\|_1))\\
\leq& C(\delta^{1/2}\lambda+\delta\lambda^{2-\tau})\leq C\delta^{1/2}\lambda.
\end{split}
\end{equation}
To absorb the linear error with respect to $\rho,$ as in \cite{CaoIn2020}, we define $\psi\in C^{\infty}([0,\infty))$ as a monotone decreasing function of $\rho$ such that
\begin{equation}\label{eq-cutoff}
	 \psi(\rho) =\begin{cases} \,\,{\rho}^{-1} \quad \, \text{ if } \rho\geq 2\epsilon^{1/2},\\
\epsilon^{-1/2} \text{ if } \rho \leq \epsilon^{1/2}\,,\end{cases}
\end{equation}
which satisfies
\[ \|\psi(\rho(\cdot))\|_0   \leq C\epsilon^{-1/2} \,,\quad \|\psi(\rho(\cdot))\|_1 \leq C\epsilon^{-1}\delta^{1/2}\lambda\,,\]
deriving from \eqref{eq-stage-rho}, and
\begin{equation}\label{eq-rholambda-k}
\| \psi(\rho) \Lambda_k \|_0 \leq C \epsilon^{-1/2}\delta^{1/2} \lambda^{1-\tau}\,,\quad
\|\psi(\rho) \Lambda_k\|_1 \leq 
C\epsilon^{-1/2}\delta^{1/2}\lambda \,,
\end{equation}
where
\[\epsilon^{1/2}= C_0(\gamma, \sigma_1) \delta^{1/2}\lambda^{1-\tau}\,,\]
and $C_0(\gamma,\sigma_1)\geq 1$ is chosen such that
$$\|H\|_0+\sum_{k=1}^{\frac n2}\|\psi(\rho)\Lambda_k\|_0+\sum_{k, l=1}^{\frac n2}\|\Theta_{kl}\|_0\leq\frac{\sigma_1}{2}+C\epsilon^{-1/2}\delta^{1/2}\lambda^{1-\tau} < \sigma_1.$$
We can also take $\lambda_*$ large enough to get $\epsilon^{1/2}\leq\delta^{1/2}$.  Lemma \ref{le-perturb} can be applied with $P_0 = G$, $P=G+H$, $N_0=n/2$ to get $n_*$ functions $\{a_i\}_{i=1}^{n_*}\subset C^{1}(\bar\Omega)$ such that
\begin{equation*}
G+H=\sum_{i=1}^{n_*}a_i^2\xi_i\otimes\xi_i+\sum_{k=1}^{\frac n2}a_k\psi(\rho)\Lambda_k+\sum_{k,l=1}^{\frac n2}a_ka_k\Theta_{kl}\,,
\end{equation*}
and then
\begin{equation}\label{eq-rho-g-h-deco}
\rho^{2}(G+H)=\sum_{i=1}^{n_*}(\rho a_i)^2\xi_i\otimes\xi_i+\sum_{k=1}^{\frac n2}\rho^{2}\psi(\rho) a_k \Lambda_k+\sum_{k,l=1}^{\frac n2}(\rho a_k) (\rho a_l)\Theta_{kl}\,.
\end{equation}
By \eqref{eq-aicj}, \eqref{eq-LTkl} and \eqref{eq-rholambda-k},  we have  for $i=1, \cdots, n_*,$
\begin{equation}\label{eq-ai-bound}
\begin{split}
0\leq a_i
\leq C,\quad 
\|a_i\|_1
\leq C\epsilon^{-1/2}\delta^{1/2}\lambda.
\end{split}
\end{equation}
Moreover, from \eqref{eq-a_i} and \eqref{eq-rholambda-k}, we get sharper estimate
\begin{align*}
&\|\rho \nabla a_i \|_0 \leq C\|\rho\nabla(\psi(\rho)\Lambda_k)\|_0
\leq C\left (\|\rho\psi'(\rho) \Lambda_k\nabla\rho  \|_0+\|\rho\psi(\rho)\nabla\Lambda_k\|_0\right )
\leq C\delta^{1/2}\lambda\,,
\end{align*}
where we have used that $|\rho\psi'(\rho)|\leq C\epsilon^{-1/2}$ and $|\rho\psi(\rho)|\leq C.$ Then it holds that
\begin{equation}\label{eq-ai-bound2}
 \| \rho a_i \|_1 \leq C(\|a_i\nabla\rho\|_0+\|\rho\nabla a_i\|_0)\leq C\delta^{1/2}\lambda \,.
\end{equation}
 For any $i=1, \cdots, n_*,$ set $b_i := \rho a_i$ and mollify $b_i$ at length scale  $\lambda^{1-2\tau}$ to get $\tilde b_i$, satisfying that for any $j\in \N$
\begin{equation}\label{eq-tilde-bk}
\begin{split}
&\| \tilde b_i \|_0 \leq  C \delta^{1/2},\quad \|\tilde b_i -b_i \|_{0} \leq C \delta^{1/2}\lambda^{2-2\tau}\,,\quad
\|\tilde b_i \|_{j+1} \leq  C_j \delta^{1/2}\lambda^{(2\tau-1)j +1},
\end{split}
\end{equation}
deriving from \eqref{eq-ai-bound} and \eqref{eq-ai-bound2}. To add the first $n/2$ primitive metrics, we may  define the first embedding as
 $$
 u_1=u+\frac{1}{\lambda^{\tau}}\sum_{k=1}^{\frac n2}\tilde b_k D_k.
 $$
From the definition it is clear that
\begin{equation}\label{eq-u1-support}
u_1=u\text{ on }\Omega\setminus \left (\supp \rho + \B_{\lambda^{1-2\tau}}\right )=\Omega\setminus \left (\supp \rho + \B_{\lambda^{-\kappa}}\right ).
\end{equation}
A straightforward calculation yields
$$\nabla u_1=\nabla u+\sum_{k=1}^{\frac n2}\tilde b_k A_k+\frac{1}{\lambda^{\tau}}\sum_{k=1}^{\frac n2}\tilde b_k B_k+
\frac{1}{\lambda^{\tau}}\sum_{k=1}^{\frac n2}D_k\otimes \nabla \tilde b_k  \,.$$
 Since $A_k^TD_k=0$ for $k=1, 2, \cdots, n/2$,  the induced metric will be
 \begin{align*}
 \nabla u_1^T\nabla u_1=&\nabla u^T\nabla u+\sum_{k=1}^{\frac n2}\tilde b_k^2\xi_k\otimes\xi_k
 +2\sum_{k=1}^{\frac n2}\tilde b_k\textrm{sym}\left (\nabla u^{T} A_k +\frac{1}{\lambda^{\tau}}\nabla u^{T}B_k\right ) \\
  & \,+ \frac{2 }{\lambda^{\tau}}\sum_{k=1}^{\frac n2}\textrm{sym}\left (\nabla u^{T}D_k\otimes \nabla\tilde b_k\right )+ \frac{2 }{\lambda^{\tau}}\sum_{k,l=1}^{\frac n2}\tilde b_k\tilde b_l\mathrm{sym}\left (A_k^{T}B_l\right )
 \\& + \frac{2}{\lambda^{2\tau}}\sum_{k,l=1}^{\frac n2}\tilde b_k\tilde b_l \mathrm{sym}\left (B_k^{T}B_l\right )
+\frac{2}{\lambda^{2\tau}}\sum_{k,l=1}^{\frac n2}\tilde b_k \mathrm{sym}\left (B_k^{T}D_l\otimes \nabla \tilde b_l\right ) \\
& +\frac{1}{\lambda^{2\tau}} \sum_{k=1}^{\frac n2} \nabla \tilde b_k \otimes \nabla \tilde b_k \,.
 \end{align*}
 Hence by \eqref{eq-rho-g-h-deco}, the first metric error will be
 \begin{align*}
 \nabla u_1^T\nabla u_1-\left(\nabla u^T\nabla u+\rho^2(G+H)
 -\sum_{k=1+\frac n2}^{n_*} b_k^2\xi_k\otimes\xi_k\right)=\mathcal{E}_{1}+\mathcal{E}_{2},
 \end{align*}
 with
 \begin{align*}
 \mathcal{E}_{1}:=&\sum_{k=1}^{\frac n2}(\tilde b_k^{2}-b_k^{2})\xi_k\otimes\xi_k
 +\sum_{k=1}^{\frac n2}(\tilde b_k -\rho\psi(\rho)b_k)\Lambda_k+\sum_{k,l=1}^{n_*}(\tilde b_k\tilde b_l - b_kb_l)\Theta_{kl},\\
 \mathcal{E}_{2}:=& \frac{2}{\lambda^{\tau}}\sum_{k=1}^{\frac n2}\textrm{sym}\left (\nabla u^{T}D_k\otimes \nabla\tilde b_k\right ) + \frac{2}{\lambda^{2\tau}}\sum_{k,l=1}^{\frac n2}\tilde b_k \mathrm{sym}\left (B_k^{T}D_l\otimes \nabla \tilde b_l\right )+\frac{1}{\lambda^{2\tau}} \sum_{k=1}^{\frac n2} \nabla \tilde b_k \otimes \nabla \tilde b_k\,.
 \end{align*}
To bound the first error, we shall note that
 \begin{equation}\label{eq-bkmollify}
 \|\tilde b_k^{2}- b_k^{2}\|_0\leq \| \tilde b_k + b_k\|_0\|\tilde b_k -b_k\|_0 \leq C \delta \lambda^{2-2\tau}\,,
 \quad \|\tilde b_k^{2}- b_k^{2}\|_1 \leq C \delta\lambda\,,
 \end{equation}
which can be derived from \eqref{eq-tilde-bk}. Similarly, it holds that
  \begin{align*}
  \|(\tilde b_k\tilde b_l - b_kb_l)\Theta_{kl}\|_0\leq C \delta \lambda^{2-2\tau},\quad
  \|(\tilde b_k\tilde b_l - b_kb_l)\Theta_{kl}\|_1 \leq C\delta\lambda \,.
  \end{align*}
 For the second term in $\mathcal E_{1}$, with the help of the facts that $1-\rho\psi(\rho) = 0$ for $\rho\geq 2\epsilon^{1/2}$ and $|b_k|\leq C\epsilon^{1/2}$ on  $\supp(1-\rho\psi(\rho)))$, using $|\nabla(\rho\psi(\rho))|\leq C\epsilon^{-1/2}\delta^{1/2}\lambda+ C|\psi(\rho)\nabla\rho |\leq C\epsilon^{-1/2}\delta^{1/2}\lambda$, \eqref{eq-tilde-bk} and \eqref{eq-LTkl}, we have
  \begin{align*}
  \|\left (\tilde b_k -\rho\psi( \rho) b_k\right )\Lambda_k \|_0 &\leq C\|\Lambda_k\|_0\left (\|\tilde b_k-b_k\|_0 + \|b_k(1-\rho\psi(\rho))\|_0 \right )
  \leq C\delta \lambda ^{2-2\tau}\\
 \|\left (\tilde b_k - \rho\psi(\rho)b_k\right )\Lambda_k \|_1 &\leq C\|\Lambda_k\|_1\left (\|\tilde b_k-b_k\|_0 + \|b_k(1-\rho\psi(\rho))\|_0 \right )\\
 &\quad+C\|\Lambda_k\|_0\left (\|\tilde b_k-b_k\|_1 + \|b_k(1-\rho\psi(\rho))\|_1\right )\leq C\delta\lambda\,,
 \end{align*}
Putting  the previous estimates together, we get
 \begin{equation}\label{eq-error1}
  \|\mathcal E_{1} \|_0 \leq C \delta\lambda^{2-2\tau}=C\delta\lambda^{1-\kappa},\,\quad
  \|\mathcal E_{1}\|_1 \leq C\delta \lambda\,.
  \end{equation}
 Directly from \eqref{eq-abdk}, \eqref{eq-udk}, \eqref{eq-tilde-bk}, we have
 \begin{equation}\label{eq-error2}
 \|\mathcal E_{2}\|_0 \leq C\delta\lambda^{1-\kappa},\,\quad \|\mathcal E_{2}\|_1 \leq C\delta \lambda\,.
 \end{equation}
This in turn implies
$${\bar\gamma}_1^{-1}\mathrm{Id}\leq \nabla u_1^T\nabla u_1\leq\bar\gamma_1\mathrm{Id},$$
for some $\bar\gamma_1$ depending only on $\gamma$ and $n$, in the same way as Proposition \ref{pr-step}.
Besides, we also have
\begin{equation}\label{eq-u1}
\begin{split}
&\|u_1-u\|_0\leq \sum_{k=1}^{\frac n2}\|\tilde b_k D_k\|_0\lambda^{-\tau }\leq C(\gamma)\delta^{1/2}\lambda^{-\tau },\\
&\|u_1-u\|_1\leq \sum_{k=1}^{\frac n2}\|\tilde b_k D_k\|_1\lambda^{-\tau }\leq C(\gamma)\delta^{1/2},\\
&\|u_1\|_2\leq\|u\|_2+\sum_{k=1}^{\frac n2}\|\tilde b_k D_k\|_2\lambda^{-\tau}\leq C(\gamma)\delta^{1/2}\lambda^\tau.
\end{split}
\end{equation}

For the remained $n\cdot n/2$ primitive metrics $\sum_{k=1+\frac n2}^{n_*}\tilde b_k^2\xi_k\otimes\xi_k$, we can apply $n/2$ times  Proposition \ref{pr-step} to add them. Note that all the conditions of Proposition \ref{pr-step} are satisfied for $u_1,$  $\tilde b_k,$ and $\Phi_k=x\cdot\xi_k$ with 
$$m=2n,\quad \delta_1=\delta,\quad \nu_1=\tilde\nu_1=\lambda^\tau.$$
Take $\delta_*=c_0^{-1}$ as in Proposition \ref{pr-step}.
Successive application of Proposition \ref{pr-step} yield $\{u_l\}_{l=2,\cdots, 1+\frac{n}{2}}$ with constants
\begin{align*}
\delta_l=\delta,\quad\nu_{l}=\tilde\nu_l=\lambda_{l},\quad \lambda_l=\lambda^\tau K^{l-1}, \,K=\lambda^{\kappa-1}, \quad \gamma_l=\bar\gamma_{l-1},
\end{align*}
and for $l=2,\dots, 1+\frac{n}{2}$
\begin{equation}\label{eq-stepjth-even}
\begin{split}
&\|u_{l}-u_{l-1}\|_0\leq C(\gamma_l)\delta^{1/2}\lambda_l^{-1}\\
&\|u_{l}-u_{l-1}\|_1\leq C(\gamma_l)\delta^{1/2},\quad \|u_{l}\|_2\leq C(\gamma_l)\delta^{1/2}\lambda_j,\\
&\left\|\nabla u_{l}^T\nabla u_{l}-\left(\nabla u_{l-1}^T\nabla u_{l-1}+ \sum_{k=n(l-2)+\frac n2}^{n(l-1)+\frac n2}\tilde b_k^2\xi_k\otimes\xi_k\right)\right\|_j\\
&\quad \leq C(\gamma_l)(\delta\lambda_{l-1}\lambda_l^{j-1}+\delta^2\lambda_l^j),\quad j=0, 1.
\end{split}
\end{equation}
Set $v=u_{1+\frac{n}{2}}\in C^2(\Omega,\R^{2n})$, which is the desired embedding. In fact, the new metric error
\begin{align*}
\mathcal{E}&=\nabla v^T\nabla v-(\nabla u^T\nabla u+\rho^2(G+H))\\
&=\mathcal{E}_3+\sum_{k=1+\frac n2}^{n_*}\tilde b_k^{2}\xi_k\otimes\xi_k+\nabla u_1^T\nabla u_1-(\nabla u^T\nabla u+\rho^2(G+H))\\
&=\mathcal{E}_3+\mathcal{E}_4+\mathcal E_{1}+\mathcal E_{2},
\end{align*}
where
\begin{align*}
\mathcal{E}_3&:=\nabla v^T\nabla v-\left(\nabla u_1^T\nabla u_1+\sum_{k=1+\frac n2}^{n_*}\tilde b_k^2\xi_k\otimes\xi_k\right)\\
&=\sum_{l=2}^{1+\frac{n}{2}}\left(\nabla u_l^T\nabla u_l-\left(\nabla u_{l-1}^T\nabla u_{l-1}+\sum_{k=nl-\frac {3n}2+1}^{nl-\frac n2}\tilde b_k^2\xi_k\otimes\xi_k\right)\right),\\
\mathcal{E}_4&:=\sum_{1+\frac n2}^{n_*}(\tilde b_k^{2}-b_k^2)\xi_k\otimes\xi_k.
\end{align*}
It follows from \eqref{eq-stepjth-even} that
\begin{equation}\label{eq-error3}
\begin{split}
\|\mathcal{E}_3\|_0&\leq \sum_{l=2}^{1+\frac n2}C(\gamma_l)(\delta\lambda_{l-1}\lambda_l^{-1}+\delta^2)\leq C(\gamma)(\delta\lambda^{1-\kappa}+\delta^2),\\
\|\mathcal{E}_3\|_1&\leq \sum_{l=2}^{1+\frac n2}C(\gamma_l)(\delta\lambda_{l-1}+\delta^2\lambda_{l})\leq
C(\gamma)(\delta\lambda^{(\frac{n}{2}-1)(\kappa-1)+\tau}+\delta^2\lambda^{\frac{n}{2}(\kappa-1)+\tau}),
\end{split}
\end{equation}
Moreover, same as the first part of $\mathcal{E}_{1},$ we have
\begin{equation}\label{eq-error4}
\|\mathcal{E}_4\|_0\leq C\delta\lambda^{2-2\tau},\quad \|\mathcal{E}_4\|_1\leq C\delta\lambda.
\end{equation}
Hence by $\tau=(\kappa+1)/2$ and \eqref{eq-error1},\eqref{eq-error2},\eqref{eq-error3} and \eqref{eq-error4},
we get \eqref{eq-stage-error-1}.  Besides, from Proposition \ref{pr-step}, it follows that
\begin{equation*}
v=u_{1+\frac n2}=\cdots =u_2=u_1  \text{ on } \Omega\setminus \left (\supp \rho + B_{\lambda^{-\kappa}}\right ),
\end{equation*}
which combining with \eqref{eq-u1-support} implies \eqref{eq-stage-1}. Moreover, with \eqref{eq-u1} and \eqref{eq-stepjth-even}, we deduce
\begin{align*}
\|v-u\|_0\leq&\|u_1-u\|_0+\sum_{l=2}^{1+\frac{n}{2}}\|u_l-u_{l-1}\|_0\leq C\delta^{1/2}\lambda^{-\tau}=C\delta^{1/2}\lambda^{-(\kappa+1)/2},\\
\|v-u\|_1\leq&\|u_1-u\|_1+\sum_{l=2}^{1+\frac{n}{2}}\|u_l-u_{l-1}\|_1\leq C\delta^{1/2},\\
\|v\|_{2}\leq& C\delta^{1/2}\lambda^{\frac{n+1}{2}(\kappa-1)+1},
\end{align*}
from which \eqref{eq-stage-v-1} follows.

\end{proof}


\section{Proof of Proposition \ref{pr-inductive}}\label{se-inductive}

The proof of Proposition \ref{pr-inductive} is similar to that of Proposition 4.1 in \cite{CaoSze2022} and is divided into three subsections.

\subsection{Parameters and cut-off functions}
First of all, recall from our global setup \eqref{e:gamma0} that in any local chart $\Omega_k$ the coordinate expression $G=(G_{ij})$ of the metric $g$ satisfies
$$
{\gamma}^{-1}\mathrm{Id}\leq G\leq \gamma\mathrm{Id},\quad \|G\|_{C^1(\Omega_k)}\leq \gamma.
$$
Since $u$ is an adapted short embedding with $\rho\leq\frac14$ (from  \eqref{eq-adapt} with $A$ is sufficiently large), after replacing $\gamma$ by $4\gamma$ if necessary we may assume in addition
$$
{\gamma}^{-1}\mathrm{Id}\leq \nabla u^T\nabla u\leq \gamma\mathrm{Id},
$$
Now, set $\delta_1:=A^{-\beta}$, and for $q\geq 1$
\begin{equation}\label{e:deltaqlambdaq}
\lambda_q=A\delta_q^{-\frac{1}{2\theta}},\quad \lambda_{q+1}=\lambda_q^b.
\end{equation}
We assume that $A$ is sufficiently large (depending on $\theta,\alpha$) so that
\begin{equation}\label{eq-ordering}
\delta_{q+1}\leq \tfrac{1}{4}\delta_q,\quad \lambda_{q+1}\geq 2\lambda_q.
\end{equation}

Next, following \cite{CaoSze2022}, we decompose $\mathcal{M}^n$ with respect to $\Sigma$ and $S$ and define associated cut-off functions. We define for $q=0, 1,2,\dots$
\begin{align*}
\Sigma_q&=\{x:\,\textrm{dist}(x,\Sigma)<s_*r_q\},\\
\widetilde\Sigma_q&=\{x:\,\textrm{dist}(x,\Sigma)<\tilde s_*r_q\},\\
S_q&=\{x:\,\textrm{dist}(x,S)<s_{**}r_q\},
\end{align*}
where
$$
r_q=A^{-1}\delta_{q+1}^{\frac{1}{2\theta}}=\lambda_{q+1}^{-1},
$$
$s_*<\tilde s_{*}$ and $s_{**}$ are geometric constants to be chosen in the following order:
\begin{enumerate}
\item[1)] Choose $s_{**}>0$ such that
\begin{equation}\label{eq-choiceofr**}
\rho(x)>\tfrac{3}{2}\delta_{q+2}^{1/2}\quad\textrm{ implies }\quad x\notin S_{q+1}.
\end{equation}
We recall that this is possible since \eqref{eq-adapt} implies $\rho(x)\leq A^{\theta}\textrm{dist}(x,S)^\theta$, and therefore $\rho(x)\leq s_{**}^\theta\delta_{q+2}^{1/2}$ for any $x\in S_{q+1}$; see \cite{CaoSze2022}.

\item[2)] Set $\tilde s_*=\bar{r}s_{**}$, where $\bar{r}>0$ is the constant in Condition \ref{c:geometric}, which yields that, for any $q\in\N$
\begin{equation}\label{eq-singlecharts}
\begin{split}
	\widetilde\Sigma_q\setminus S_q&\textrm{ \emph{is contained in a pairwise disjoint union of open sets,}}\\
	&\textrm{\emph{each contained in a single chart }}\Omega_k	
\end{split}
\end{equation}
\item[3)] Choose $s_*<\tilde s_*$ so that $\tfrac{1}{2}\tilde s_*<s_*<\tilde s_*$, hence 
$$
\widetilde\Sigma_{q+1}\subset \Sigma_q\subset \widetilde\Sigma_q\quad\textrm{ for all }q.
$$
\end{enumerate}
We then define cut-off functions $\chi_q\leq \tilde\chi_q$ associated to the sets $\Sigma_q\subset \tilde\Sigma_q$ as follows. Define $\phi, \tilde\phi, \psi, \tilde\psi\in C^\infty(0,\infty)$ with  $\phi,\tilde\phi$ monotonic increasing and $\psi,\tilde\psi$ monotonic decreasing, such that
$$
\phi(t),\,\tilde\phi(t)=\begin{cases} 1&t\geq 2\\ 0&t\leq \tfrac32\end{cases}\,,\quad
\tilde\phi(t)=1 \textrm{ on }\supp\phi,
$$
and
$$
\psi(t),\,\tilde\psi(t)=\begin{cases} 1&t\leq s_*\\ 0&t\geq \tilde s_*\end{cases}\,,
\quad \tilde\psi(t)=1 \textrm{ on }\supp\psi.
$$
Set
\begin{align*}
\chi_q(x)=\phi\left(\frac{\rho(x)}{\delta_{q+2}^{1/2}}\right)\psi\left(\frac{\textrm{dist}(x, \Sigma)}{r_{q+1}}\right),\quad
\tilde\chi_q(x)=\tilde\phi\left(\frac{\rho(x)}{\delta_{q+2}^{1/2}}\right)\tilde\psi\left(\frac{\textrm{dist}(x, \Sigma)}{r_{q+1}}\right).
\end{align*}
From \eqref{eq-adapt} and the choice of $r_q$, $s_*$, $\tilde s_*$ and the cut-off functions, the following estimates can be obtained
 \begin{align}
 |\nabla\chi_q|,\, |\nabla\tilde\chi_q|&\leq CA\delta_{q+2}^{-\frac{1}{2\theta}}=C\lambda_{q+2},\label{eq-gradient-chi}\\
\textrm{dist}(\supp\chi_q, \partial\supp\tilde\chi_q)&\geq C^{-1}A^{-1}\delta_{q+2}^{\frac{1}{2\theta}}=C^{-1}\lambda^{-1}_{q+2},\label{eq-chi-q-support}
 \end{align}
where the constant $C$ depends on $s_*,\tilde s_*$.Moreover, it holds that
\begin{equation}\label{eq-chi-define}
\begin{split}
\{x\in\Sigma_{q+1}|\rho(x)>2\delta_{q+2}^{1/2}\}&\subset\{x\in\mathcal{M}^n:\,\chi_q(x)=1\},\\
\supp\chi_q&\subset \{x\in\mathcal{M}^n:\,\tilde{\chi}_q(x)=1\},\\
\supp\tilde{\chi}_q&	\subset\{x\in\widetilde\Sigma_{q+1}:\,\rho(x)>\tfrac{3}{2}\delta_{q+2}^{1/2}\}.
\end{split}
\end{equation}
\subsection{Inductive construction of a sequence of adapted short embeddings}
Following \cite[Section 4.4]{CaoSze2022}, we define the sequence of metric difference size $\{\rho_q\}$ in the following way.
Set $\rho_0=\rho$ and define $\rho_{q}$ for $q=1,2,\dots$ as
\begin{equation}\label{eq-rho-q+1}
\rho_{q+1}^2=\rho_q^2(1-\chi_q^2)+\delta_{q+2}\chi_q^2,
\end{equation}
which satisfies the following claims in \cite[Lemma 4.1]{CaoSze2022}: for any $q=0,1,\dots$
\begin{enumerate}
	\item[(a)] On $\supp\tilde\chi_q$ it holds that
	$\tfrac{3}{2}\delta_{q+2}^{1/2}\leq\rho_q\leq2\delta_{q+1}^{1/2}.$
	\item[(b)] For every $x$,  $\rho_{q+1}(x)\leq\rho_q(x)$.
	\item[(c)] If $\rho_q(x)\leq\delta_{q+1}^{1/2}$, then $x\not\in\bigcup_{j=0}^{q-1}\supp\tilde\chi_j$ and consequently $\rho_q(x)=\rho(x)$.
	\item[(d)] If $\rho_q(x)\geq \delta_{q+1}^{1/2}$, then either $\chi_q(x)=1$ or $x\notin \Sigma_{q+1}$.
	\end{enumerate}

Having defined the sequence $\{\rho_q\}$, we then construct inductively a sequence of smooth adapted short embeddings $u_q$ with associated metric error $h_q$ such that for any $q=0,1,\dots$, the following hold:
\begin{itemize}
\item[$(1)_q$] For all $\mathcal{M}^n,$ it holds that $g-u_q^\sharp e=\rho_q^2(g+h_q).$
\item[$(2)_q$] If $x\notin\bigcup_{j=0}^{q-1}\supp\tilde\chi_j$, then $(u_q, \rho_q, h_q)=(u_0, \rho_0, h_0).$
\item[$(3)_q$] In $\mathcal{M}^n$ it holds that
\begin{align}
|\nabla^2u_q|\leq A^{b^2}\rho_q^{1-\frac{b^2}{\theta}},\quad & |\nabla\rho_q|\leq A^{b^2}\rho_q^{1-\frac{b^2}{\theta}} ,\label{eq-inductive-v-rho-j}\\
|h_q|\leq A^{-\frac{\theta\alpha}{2b^2}}\rho_q^{\frac{\alpha}{2b^2}},\quad &|\nabla h_q|\leq A^{{b^2}-\frac{\theta\alpha}{2b^2}}\rho_q^{\frac{\alpha}{2b^2}-\frac{b^2}{\theta}},
\label{eq-inductive-h-j}
\end{align}
\item[$(4)_q$] On $\{x: \rho(x)>\delta_{q+1}^{1/2}\}\cap\Sigma_q,$  the following sharper estimates hold
\begin{align}
|\nabla^2u_q|\leq A^{b}\rho_q^{1-\frac{b}{\theta}}, \quad &|\nabla\rho_q|\leq A^b\rho_q^{1-\frac{b}{\theta}} ,\label{eq-inductive-rho-q}\\
|h_q|\leq A^{-\frac{\theta\alpha}{b}}\rho_q^{\frac{\alpha}{b}}, \quad &|\nabla h_q|\leq A^{b-\frac{\theta\alpha}{b}}\rho_q^{\frac{\alpha}{b}-\frac{b}{\theta}}.
\label{eq-inductive-h-q}
\end{align}
\item[$(5)_q$]  Globally, it holds that for $q\geq1$
\begin{align}
&\|u_q-u_{q-1}\|_{C^0(\mathcal{M})}\leq \overline{C}\delta_q^{1/2}\lambda_{q}^{-1},\label{eq-inductive-v-difference-0}\\
&\|u_q-u_{q-1}\|_{C^1(\mathcal{M})}\leq\overline{C}\delta_q^{1/2},\label{eq-inductive-v-difference-1}
\end{align}
where $\overline{C}$ is the constant in Proposition \ref{pr-stage} and the norm is taken on $\mathcal{M}^n$.
\end{itemize}

{\it Initial step $q=0$.} Set $(u_0, h_0)=(u, h)$. Property $(1)_0$ holds by assumption with global estimates sharper than \eqref{eq-inductive-rho-q}-\eqref{eq-inductive-h-q} since $b>1$ and $\rho<1$. Thus $(3)_0$ is satisfied, whereas $(2)_0$, $(4)_0$ and $(5)_0$ are empty. 

{\it Inductive step $q\mapsto q+1$.} Suppose $(u_q, h_q)$ has already been defined and properties $(1)_q-(5)_q$ hold. We will show how an application of Proposition \ref{pr-stage} leads to $(u_{q+1}, h_{q+1})$ on $\supp\tilde\chi_q$. To this end we first deduce uniform estimates for $(u_q, \rho_q, h_q)$ on $\supp\tilde\chi_q$. 

By (a) above, on $\supp\tilde\chi_q,$ it holds that
\begin{equation}\label{eq-rho-q-bound}
\tfrac{3}{2}\delta_{q+2}^{1/2}\leq\rho_q\leq2\delta_{q+1}^{1/2}.
\end{equation}
Then, as in \cite[(4.19)-(4.20)]{CaoSze2022}, we use the definition of $\rho_q$ in \eqref{eq-rho-q+1} and property (c) whenever  $\frac{3}{2}\delta_{q+2}^{1/2}\leq\rho_q(x)\leq\delta_{q+1}^{1/2}$ on the one hand, and  property $(4)_q$ whenever $\delta_{q+1}^{1/2}<\rho_q(x)\leq2\delta_{q+1}^{1/2}$ on the other hand, to obtain
on $\supp\tilde\chi_q$ the estimates
\begin{equation}\label{eq-support-rho_q}
\begin{split}
\quad |\nabla^2u_q|&\leq \delta_{q+1}^{1/2}\lambda_{q+2}\,,\\
 |\nabla \rho_q|&\leq \delta_{q+1}^{1/2}\lambda_{q+2} \,,\quad \left|\frac{\nabla\rho_q}{\rho_q}\right|\leq \lambda_{q+2}\,,\\
|h_q|&\leq 2\lambda_{q+2}^{-\frac{\theta\alpha}{b^2}}\,, \quad |\nabla h_q|\leq \lambda_{q+2}^{1-\frac{\theta\alpha}{b^2}}\,.
\end{split}
\end{equation}
In order to apply Proposition \ref{pr-stage} we define
\begin{align*}
\tilde{\rho}_q=\chi_q\sqrt{\rho_q^2-\delta_{q+2}}\,,\quad \tilde h_q=\frac{\tilde\chi_q\rho_q^2}{\rho_q^2-\delta_{q+2}}h_q,
\end{align*}
so that
$$\tilde\rho_q^2(g+\tilde h_q)=\chi_q^2(\rho_q^2(g+h_q)-\delta_{q+2}g)=\chi_q^2(g-u_{q}^\sharp e-\delta_{q+2}g).$$
Due to \eqref{eq-rho-q-bound} we have
$\tfrac54\delta_{q+2}\leq\rho_q^2-\delta_{q+2}\leq4\delta_{q+1}$ on $\supp\tilde\chi_q$. Therefore $\tilde\rho_q$ and $ \tilde h_q$ are well defined.
Furthermore, using \eqref{eq-gradient-chi} and \eqref{eq-support-rho_q}, we infer
\begin{equation}\label{eq-rho-tilde}
\begin{split}
|\nabla^2u_q|\leq &\delta_{q+1}^{1/2}\lambda_{q+2},\quad 0\leq\tilde\rho_q\leq\rho_q\leq2\delta_{q+1}^{1/2}\,,\quad |\tilde h_q|\leq2|h_q|\leq4\lambda_{q+2}^{-\frac{\theta\alpha}{b^2}},\\
|\nabla\tilde\rho_q|&\leq C(|\nabla\chi_q|\rho_q+|\nabla\rho_q|)\leq C\delta_{q+1}^{1/2}\lambda_{q+2},\\
|\nabla\tilde h_q|&\leq C(|\nabla\tilde\chi_q||h_q|+\left|\frac{\nabla\rho_q}{\rho_q}\right||h_q|+|\nabla h_q|)\leq C\lambda_{q+2}^{1-\frac{\theta\alpha}{b^2}}.
\end{split}
\end{equation}
In particular, $(u_q, \tilde\rho_q, \tilde h_q)$ satisfies the estimates \eqref{eq-stage-u}-\eqref{eq-stage-H} of Proposition \ref{pr-stage} on $\supp\tilde\chi_q$ with 
$$
\delta=4\delta_{q+1}, \lambda=C\lambda_{q+2}, 
$$
and $\alpha$ given by $\frac{\theta\alpha}{4b^2}$. From \eqref{eq-choiceofr**}  and \eqref{eq-singlecharts} it follows that  $\supp\tilde\chi_q$ is contained in a pairwise disjoint union of open sets, each contained in a single chart. Therefore, we can apply Proposition \ref{pr-stage} in local coordinates of $\supp\tilde\chi_q$ to add the term $\tilde\rho_q^2(g+\tilde h_q)$, with 
\begin{equation}\label{e:kappa}
\kappa=1+\frac{2\theta}{b}(b-1+\alpha)>1,
\end{equation}
provided $A\gg 1$ is sufficiently large (depending on $\alpha,\beta,b,\theta$ as well as $\gamma,n,\sigma_0$) so that the conditions $4\delta_{q+1}\leq 4\delta_1\leq \delta_*$ and $C\lambda_{q+2}\geq C\lambda_2\geq \lambda_*$ are satisfied. 

In this way we obtain $u_{q+1}$ and $\mathcal{E}$ such that
$$
g-u_{q+1}^\sharp e=(g-u_q^\sharp e)(1-\chi_q^2)+\delta_{q+2}g\chi_q^2+\mathcal{E}
$$
From \eqref{eq-stage-2} and \eqref{eq-stage-3} we deduce 
\begin{equation}\label{eq-uq+1-c2}
|\nabla^2u_{q+1}|\leq C\delta_{q+1}^{1/2}\lambda_{q+2}^{1+N(\kappa-1)}
=C\delta_{q+1}^{1/2}\lambda_{q+1}^{b+2N(b-1+\alpha)\theta},
\end{equation}
and
\begin{equation}\label{eq-errorq+1}
\begin{split}
|\mathcal{E}|&\leq C(\delta_{q+1}\lambda_{q+2}^{1-\kappa}+\delta_{q+1}^2)\,\\
|\nabla\mathcal{E}|&\leq C\delta_{q+1}\left(\lambda_{q+2}^{1+(N-1)(\kappa-1)}
+\delta_{q+1}\lambda_{q+2}^{1+N(\kappa-1)}\right).
\end{split}
\end{equation}
Next, using \eqref{e:deltaqlambdaq}, \eqref{e:kappa} and the choice of $b$ in \eqref{e:choiceb} we compute 
\begin{align*}
    \delta_{q+1}\lambda_{q+2}^{\kappa-1}=\delta_{q+1}\lambda_{q+1}^{2\theta(b-1+\alpha)}=A^{\frac{2\alpha\theta(1-\theta)}{1-\theta(1+2N)}}\delta_{q+1}^{\frac{1-\theta(1+2N)-\alpha(1-\theta)}{1-\theta(1+2N)}},
\end{align*}
from which we deduce that 
\begin{equation}\label{e:deltalambdakappa}
   \delta_{q+1}\leq \lambda_{q+2}^{1-\kappa}, 
\end{equation}
provided 
$\delta_{q+1}\leq \delta_1\leq A^{-\beta}$ 
for any $\beta>0$ with 
$$
\beta>\frac{2\alpha\theta(1-\theta)}{1-(1+2N)\theta-\alpha(1-\theta)}.
$$
In particular this holds if $\alpha<\frac{1-\theta(1+2N)}{2(1-\theta)}$ and $\beta>\frac{4\alpha\theta(1-\theta)}{1-\theta(1+2N)}$, or expressed differently, whenever $\alpha,\beta$ satisfy
\begin{equation}\label{e:alphabeta}
\alpha<\min\left\{\tfrac{1-\theta(1+2N)}{2(1-\theta)},\tfrac{1-\theta(1+2N)}{4\theta(1-\theta)}\beta\right\}\leq c_*(N,\theta)\beta, 
\end{equation}
with
\begin{equation}\label{e:c*}
c_*(N,\theta):=\tfrac{1-\theta(1+2N)}{4\theta(1-\theta)}.
\end{equation}
Observe that here we also require $\theta<(1+2N)^{-1}$ and $\beta\leq 1$.
This determines the choice of parameters in \eqref{e:alpha*beta}.
In turn, with these choice of parameters \eqref{e:deltalambdakappa} implies that we can drop the second terms in the estimates for $\mathcal{E}$ and $\nabla\mathcal{E}$ in \eqref{eq-errorq+1}, to arrive at 
\begin{align}
|\mathcal{E}|&\leq C\delta_{q+2}\lambda_{q+1}^{-2\theta\alpha}, \label{eq-error-q+1--1}\\
|\nabla\mathcal{E}|&\leq C\delta_{q+2}\lambda_{q+1}^{b+2\theta N(b-1)
+2\theta(N-1)\alpha}.\label{eq-error-q+1--2}
\end{align}

Now we proceed exactly as in \cite{CaoSze2022}.
By \eqref{eq-stage-1}, we get
$$
\supp(u_{q+1}-u_q), \quad\supp\,\mathcal{E}\subset\supp\,\chi_q+\mathbb{B}_{\tau_q}(0),$$
with
$$
\tau_q=(C\lambda_{q+2})^{-\kappa}\leq A^{-{2\theta}(b-1+\alpha)}\lambda_{q+2}^{-1}\leq C^{-1}\lambda_{q+2}^{-1},
$$
where $C$ is the constant in \eqref{eq-chi-q-support} and the last inequality holds provided $A$ is sufficiently large. Consequently $u_{q+1}=u_q$ and $\mathcal{E}=0$ outside  $\supp\tilde\chi_q$.

Moreover, \eqref{eq-inductive-v-difference-0}-\eqref{eq-inductive-v-difference-1} for the case $q+1$ follows immediately from \eqref{eq-stage-2}, hence $(5)_{q+1}$ is verified. Define
$$h_{q+1}=(1-\chi_q^2)\frac{\rho_q^2}{\rho_{q+1}^2}h_q+\frac{\mathcal{E}}{\rho_{q+1}^2}$$
and then $g-u_{q+1}^\sharp e=\rho_{q+1}^2(g+h_{q+1}),$
verifying $(1)_{q+1}.$  Note that on $\supp\tilde\chi_q$ using \eqref{eq-rho-q-bound} we have
\begin{equation}\label{eq-rhoq+1-bound}
\begin{split}
\rho_{q+1}^2&\leq4\delta_{q+1}(1-\chi_q^2)+\delta_{q+2}\chi_q^2\leq 4\delta_{q+1},\\
\rho_{q+1}^2&\geq\tfrac94\delta_{q+2}(1-\chi_q^2)+\delta_{q+2}\chi_q^2\geq\delta_{q+2},
\end{split}
\end{equation}
 which also implies that $\mathcal{E}$ and $h_{q+1}$ are well defined. Further, it is easy to see that $(\rho_{q+1}, h_{q+1})$ agrees with $(\rho_q, h_q)$ outside $\supp\tilde\chi_q$. Therefore, to conclude with the induction step we need to verify  $(2)_{q+1}-(4)_{q+1}$ on $\supp\tilde\chi_q$.

\bigskip

{\it Verification of $(2)_{q+1}$.} If $x\not\in\bigcup_{j=0}^{q}\supp\tilde\chi_j$, then  $\tilde\chi_q(x)=0$ and therefore
$$(u_{q+1}, \rho_{q+1}, h_{q+1})=(u_q, \rho_q, h_q)=(u_0, \rho_0, h_0).$$

{\it Verification of $(3)_{q+1}$.} 

To verify the estimates in $(3)_{q+1}$ on $\supp\tilde\chi_q$ we can use the bound
\begin{equation}\label{e:rhoq1bound1}
\delta_{q+2}^{1/2}\leq \rho_{q+1}\leq 2\delta_{q+1}^{1/2},
\end{equation}
which follows from \eqref{eq-rhoq+1-bound}. 
Using \eqref{eq-gradient-chi}, \eqref{eq-rho-q+1} and \eqref{eq-support-rho_q} we obtain
\begin{equation}\label{eq-gradient-rho-q+1}
|\nabla\rho_{q+1}|=\frac{|\nabla\rho_{q{+1}}^2|}{2\rho_{q+1}}\leq\frac{C}{\rho_{q+1}}(|\rho_q\nabla\rho_q|
+|\nabla\chi_q|(\rho_q^2+\delta_{q+2}))\leq C\frac{\delta_{q+1}\lambda_{q+2}}{\delta_{q+2}^{1/2}}.
\end{equation}
Similarly, using \eqref{eq-support-rho_q}, \eqref{eq-uq+1-c2}-\eqref{eq-error-q+1--1} and \eqref{eq-rhoq+1-bound} we obtain
\begin{equation}\label{eq-hq+1-bound-v-c2}
|h_{q+1}|\leq|h_q|+\frac{|\mathcal{E}|}{\rho_{q+1}^2}\leq 2\lambda_{q+2}^{-\frac{\alpha\theta} {b^2}}+C\lambda_{q+1}^{-2\theta\alpha}
\end{equation}
and for $|\nabla h_{q+1}|$ we compute, using \eqref{eq-gradient-chi}, \eqref{eq-support-rho_q}, \eqref{eq-error-q+1--1}, \eqref{eq-error-q+1--2} and \eqref{eq-gradient-rho-q+1}, 
\begin{equation}\label{eq-gradient-h-q+1}
\begin{split}
|\nabla h_{q+1}|&\leq|\nabla h_q|+\frac{1}{\rho_{q+1}^2}(|\nabla\mathcal{E}|+\delta_{q+2}|\nabla(h_q\chi_q^2)|)
+\frac{2|\nabla\rho_{q+1}|}{\rho_{q+1}^3}(\delta_{q+2}|h_q|+|\mathcal{E}|)\\
&\leq C\lambda_{q+2}^{1-\frac{\theta\alpha}{b^2}}
+C\left(\lambda_{q+1}^{b+2\theta N(b-1)
+2\theta(N-1)\alpha}+\lambda_{q+2}^{1-\frac{\theta\alpha}{b^2}}\right)
+C\frac{\delta_{q+1}\lambda_{q+2}}{\delta_{q+2}}
(\lambda_{q+2}^{\frac{-\theta\alpha}{b^2}}+\lambda_{q+1}^{-2\theta\alpha})\\
&\leq C\lambda_{q+1}^{b+2\theta N(b-1)
+2\theta(N-1)\alpha},
\end{split}
\end{equation}

\bigskip

Next, we claim that $\textrm{\eqref{eq-inductive-v-rho-j}}_{q+1}$ will follow from
\begin{equation}\label{e:3-u-rho}
|\nabla^2 u_{q+1}|,|\nabla \rho_{q+1}|\leq c\delta_{q+1}^{1/2}\lambda_{q+1}^{b^2},
\end{equation}
for some small geometric constant $c$.
Indeed, using \eqref{e:deltaqlambdaq}, 
$$
\delta_{q+1}^{1/2}\lambda_{q+1}^{b^2}=A^{b^2}\delta_{q+1}^{\tfrac12-\tfrac{b^2}{2\theta}}\leq CA^{b^2}\rho_{q+1}^{1-\frac{b^2}{\theta}}.
$$
Similar computations show that $\textrm{\eqref{eq-inductive-h-j}}_{q+1}$ will follow from 
\begin{equation}\label{e:3-h}
|h_{q+1}|\leq c\lambda_{q+2}^{-\frac{\theta\alpha}{2b^2}},\quad |\nabla h_{q+1}|\leq c\lambda_{q+1}^{b^2-\frac{\theta\alpha}{2b^2}}.
\end{equation}
Thus, in order to conclude $(3)_{q+1}$, we just need to compare the powers of $\lambda_{q+1}$ in the estimates \eqref{eq-gradient-rho-q+1}, \eqref{eq-uq+1-c2}, \eqref{eq-error-q+1--1}, \eqref{eq-error-q+1--2} to the powers in \eqref{e:3-u-rho} and \eqref{e:3-h}: we need
\begin{equation}\label{e:bNalpha1}
    b+2N\theta(b-1+\alpha)<b^2,\quad b+2\theta N(b-1)+2\theta(N-1)\alpha<b^2-\frac{\theta\alpha}{2b^2}
\end{equation}
as well as 
$$
b+\theta(b-1)< b^2,\quad -\frac{\theta\alpha}{b^2}< -\frac{\theta\alpha}{2b^2}.
$$
The latter two are obviously satisfied, whereas the former two can be checked using the choice of $b$ in \eqref{e:choiceb}. Then, for $A$ sufficiently large (in order to take care of the small geometric constants in \eqref{e:3-u-rho}-\eqref{e:3-h}), the estimates in $(3)_{q+1}$ hold.

{\it Verification of $(4)_{q+1}$.} Observe that
\begin{align*}
\{x\in\Sigma_{q+1}: \rho_0(x)>\delta_{q+2}^{1/2}\}
=\{\chi_q(x)=1\}\cup\{x\in\Sigma_{q+1}: \delta_{q+2}^{1/2}\leq\rho_0(x)\leq2\delta_{q+2}^{1/2}\}.
\end{align*}
Moreover, if $x\in\{\chi_q=1\},$ then $\rho_{q+1}=\delta_{q+2}^{1/2}$.
Arguing analogously to above, but this time using the bound
\begin{equation}\label{e:rhoq1bound2}
\delta_{q+2}^{1/2}\leq \rho_{q+1}\leq 2\delta_{q+2}^{1/2},
\end{equation}
we see that $\textrm{\eqref{eq-inductive-rho-q}}_{q+1}$ will follow from
\begin{equation}\label{e:4-u-rho}
|\nabla^2 u_{q+1}|,\,|\nabla \rho_{q+1}|\leq c\delta_{q+2}^{1/2}\lambda_{q+2}^{b}(=c\delta_{q+1}^{1/2}\lambda_{q+1}^{b^2-\theta(b-1)}),
\end{equation}
whereas $\textrm{\eqref{eq-inductive-h-q}}_{q+1}$ will follow from
\begin{equation}\label{e:4-h}
    |h_{q+1}|\leq c\lambda_{q+2}^{-\frac{\theta\alpha}{b}},\quad |\nabla h_{q+1}|\leq c\lambda_{q+2}^{b-\frac{\theta\alpha}{b}}(=c\lambda_{q+1}^{b^2-\theta\alpha}).
\end{equation}

On the other hand, repeating the computations leading to \eqref{eq-gradient-rho-q+1}, but this time using \eqref{e:rhoq1bound2} instead of \eqref{e:rhoq1bound1}, we obtain
\begin{equation}\label{e:4-rhoq1}
    |\nabla \rho_{q+1}|\leq C\delta_{q+2}^{1/2}\lambda_{q+2}.
\end{equation}
Next, we estimate $h_{q+1}$, considering two cases. If $\chi_q(x)=1$, then $h_{q+1}=\frac{\mathcal{E}}{\delta_{q+2}}$, therefore using \eqref{eq-error-q+1--1},
\begin{equation}\label{e:4-hq1-1}
    |h_{q+1}|\leq C\lambda_{q+1}^{-2\theta\alpha}
    =C\lambda_{q+2}^{-\frac{2\theta\alpha}{b}}. 
\end{equation}
On the other hand, if $x\in\{x\in\Sigma_{q+1}: \delta_{q+2}^{1/2}\leq\rho_0(x)\leq2\delta_{q+2}^{1/2}\},$
then $(u_q, \rho_q, h_q)=(u_0, \rho_0, h_0)$ by $(2)_q$.
Thus
\begin{equation}\label{eq-h-q+1-bound}
\begin{split}
|h_{q+1}|&\leq |h_0|+\left|\frac{\mathcal{E}}{\rho_{q+1}^2}\right|
\leq C(\lambda_{q+2}^{-2\alpha\theta}+\lambda_{q+2}^{-\frac{2\theta\alpha}{b}})
\leq C\lambda_{q+2}^{-\frac{2\alpha\theta}{b}}.
\end{split}
\end{equation}
Again comparing powers of $\lambda_{q+1}$ in \eqref{e:4-u-rho}-\eqref{e:4-h} to the estimates for $\nabla\rho_{q+1}, \nabla u_{q+1}, h_{q+1}$ and $\nabla h_{q+1}$ in \eqref{eq-uq+1-c2} and \eqref{eq-gradient-h-q+1}, we reduce to the inequalities
\begin{equation}\label{e:bNalpha2}
    b+2N\theta(b-1+\alpha)<b^2-\theta(b-1),\quad b+2\theta N(b-1)+2\theta(N-1)\alpha<b^2-\theta\alpha,
\end{equation}
and comparing powers of  $\lambda_{q+2}$ in \eqref{e:4-u-rho}-\eqref{e:4-h}  to that in  \eqref{e:4-rhoq1}, \eqref{e:4-hq1-1} and  \eqref{eq-h-q+1-bound}, we require
\begin{equation*}
    1<b,\quad -2\theta\alpha<-\theta\alpha,
\end{equation*}
which are obvious. \eqref{e:bNalpha2} can again be checked using the choice of $b$ in \eqref{e:choiceb} (and are in fact sharper than the inequalities \eqref{e:bNalpha1}.   
With this, we thus conclude the verification of $(4)_{q+1}$.  

This concludes the proof of the induction step $(1)_{q+1}-(5)_{q+1}$.

\subsection{Conclusion}

We can see $\delta_q^{1/2}\leq 2^{-q-1}$ and $\delta_q^{1/2}\lambda_q^{-1}\leq A^{-1}2^{-q-1}$ from \eqref{eq-ordering}. With $(5)_q$,  we will see that $\{u_q\}$ is a Cauchy sequence in $C^1(\mathcal{M}^n)$.
From \eqref{eq-rho-q+1}, we have $0\leq \rho_{q}-\rho_{q+1}\leq 2\delta_{q+1}^{1/2}$, so that $\{\rho_q\}$ is a Cauchy sequence in $C^0(\mathcal{M}^n)$. From $(1)_q-(3)_q$, it follows that $\{h_q\}$ is a Cauchy sequence in $C^0(\mathcal{M}^n)$. 
Note that $\supp\tilde\chi_q\subset\Sigma_q$ and $\cap_q\Sigma_q=\Sigma$. We obtain from $(2)_q$ that  for any $x\in \mathcal{M}^n\setminus \Sigma$ there exists $q_0=q_0(x)$ such that
$
(u_{q}, \rho_{q}, h_{q})=(u_{q_0}, \rho_{q_0}, h_{q_0})
$
for all $q\geq q_0(x)$. From $\supp\tilde\chi_q\subset\{\rho>\delta_{q+1}^{1/2}\}$, $(u_q, \rho_q, h_q)$ agrees with $(u,\rho,h)$ on $S$. Thus there exist
\begin{align*}
\bar{u}&\in C^1(\mathcal{M}^n)\cap C^2(\mathcal{M}^n\setminus \Sigma),\\
\bar{\rho}&\in C^0(\mathcal{M}^n)\cap C^1(\mathcal{M}^n\setminus\Sigma),\\
\bar{h}&\in C^0(\mathcal{M}^n, \R^{2\times2})\cap C^1(\mathcal{M}^n\setminus\Sigma, \R^{2\times 2}),
\end{align*}
such that
$$
u_q\rightarrow \bar{u}, \quad u_q^\sharp e\rightarrow \bar{u}^\sharp e, \quad \rho_q\rightarrow\bar{\rho},\quad h_q\rightarrow \bar{h}\textrm{ uniformly on }\mathcal{M}^n.
$$
First, $(\bar u, \bar\rho, \bar h)=(u,\rho,h)$ on $S$. Secondly, by $(1)_q$, the limit $(\bar{u}, \bar{\rho}, \bar{h})$ satisfies
$$
g-\bar{u}^\sharp e=\bar{\rho}^2(g+\bar{h})\textrm{ on }\mathcal{M}^n
$$
and by $(5)_q$
\[\|\bar u-u\|_0\leq\sum_{q=1}^\infty\|u_q-u_{q-1}\|_0\leq\overline CA^{-1}\sum_{q\geq1}2^{-q-1}\leq A^{-1}.\]
Thirdly, from $(3)_q$ and (b), it follows that
\begin{align*}
&|\nabla^2\bar{u}|\leq A^{b^2}\bar{\rho}^{1-\frac{b^2}{\theta}},\quad |\nabla\bar{\rho}|\leq A^{b^2}\bar{\rho}^{1-\frac{b^2}{\theta}},\\
0\leq\bar\rho\leq\rho&\leq A^{-\beta},\quad |\bar{h}|\leq A^{-\frac{\theta\alpha}{2b^2}}\bar{\rho}^{\frac{\alpha}{2b^2}},\quad |\nabla \bar{h}|\leq A^{{b^2}-\frac{\theta\alpha}{2b^2}}\bar{\rho}^{\frac{\alpha}{2b^2}-\frac{b^2}{\theta}}.
\end{align*}
Finally, from (a) and \eqref{eq-chi-define}, we get $\rho_q\leq 2\delta_{q+1}^{1/2}$ on $\Sigma$, which then implies that 
$\{\bar{\rho}=0\}=\Sigma$.

Finally, to see that $u\in C^{1,\theta'}(\mathcal{M})$, we observe that 
$(u_{q+1}-u_q)$ is supported on $\supp\tilde\chi_{q}$; hence $|\nabla^2u_{q+1}-\nabla^2u_q|$ can be estimated using the estimate \eqref{eq-uq+1-c2} for $|\nabla^2u_{q+1}|$. Using \eqref{e:bNalpha2} we conclude
$$
\|u_{q+1}-u_q\|_{C^2(\mathcal{M})}\leq C\delta_{q+1}^{1/2}\lambda_{q+1}^{b^2-\theta(b-1)}.
$$
Combining with \eqref{eq-inductive-v-difference-1} and interpolation, we see that
\begin{align*}
\|u_{q+1}-u_q\|_{C^{1,\theta'}(\mathcal{M})}&\leq C\delta_{q+1}^{1/2}\lambda_{q+1}^{\theta'b^2-\theta'\theta(b-1)}\\
&= CA^{\theta}\lambda_{q+1}^{-\frac{\theta^2}{b^2}(b-1)}
\end{align*}
so that $(u_q)_q$ is a Cauchy sequence in $C^{1,\theta'}$ for $\theta'=\theta/b^2$.

This proves that $\bar{u}$ is a desired adapted short embedding and completes the proof of Proposition \ref{pr-inductive}.

\begin{appendix}
\section{H\"older spaces}

We recall the following interpolation inequality for H\"older norms
\begin{equation}\label{eq-interpolation}
\|f\|_{j,\alpha}\leq C\|f\|_{j_1,\alpha_1}^{\tau}\|f\|_{j_2,\alpha_2}^{1-\tau},
\end{equation}
where $C$ depends $j, j_1, j_2$, $\alpha, \alpha_1,\alpha_2,\tau\in[0, 1]$ and
$$
j+\alpha=\tau(j_1+\alpha_1)+(1-\tau)(j_2+\alpha_2).
$$

For compositions, we also have the following estimates.

\begin{proposition}\label{pr-composition}
Let $\Psi:\Omega\to\R$ and $f: \R^n\supset U\to\Omega$ be two $C^j$ functions with $\Omega\subset\R^{n_0},$
Then for any $r, s\geq0$, it holds
\begin{align*}
&\|\Psi\circ f\|_r\leq C(r)(\|\Psi(\cdot)\|_r\|f\|_1^r+\|\Psi(\cdot)\|_1\|f\|_r+\|\Psi(\cdot)\|_0),\, r\geq1,\\
&\|\Psi\circ f\|_r\leq \min(\|\Psi(\cdot)\|_r\|f\|_1^r,\, \|\Psi(\cdot)\|_1\|f\|_r)+\|\Psi(\cdot)\|_0,\, 0\leq r\leq1.
\end{align*}
Let $f_1, f_2: \R^n\supset U \to \R$ be two $C^j$ functions. Then there is a constant $C(\alpha, j, n, U)$ such that
\[ [f_1f_2]_j\leq C(\|f_1\|_0[f_2]_j + \|f_2\|_0[f_1]_j).\]
\end{proposition}

For mollification the following claim holds (see e.g. \cite[Lemma 1]{CoDeSz2012}).
\begin{proposition}\label{pr-mollification}
Let $\varphi\in C_c^\infty(B_1(0))$ be symmetric, nonnegative function such that $\int\varphi dx=1$. For any $r, s\geq0,$ and $0<\alpha\leq1,$ it holds that
\begin{itemize}
\item[(1)]$\|f*\varphi_\ell\|_{r+s}\leq C(r, s)\ell^{-s}\|f\|_r$,
\item[(2)]$\|f-f*\varphi_\ell\|_r\leq C\ell^{1-r}[f]_1, \text{ if } 0\leq r\leq1,$
\item[(3)] $\|(f_1f_2)*\varphi_\ell-(f_1*\varphi_\ell)(f_2*\varphi_\ell)\|_r\leq C(r, \alpha)\ell^{2\alpha-r}\|f_1\|_\alpha\|f_2\|_\alpha.$
\end{itemize}
\end{proposition}

\end{appendix}

%
\bigskip

\bibliographystyle{amsalpha}
\bibliography{references}

\providecommand{\bysame}{\leavevmode\hbox to3em{\hrulefill}\thinspace}
\providecommand{\MR}{\relax\ifhmode\unskip\space\fi MR }
\providecommand{\MRhref}[2]{%
  \href{http://www.ams.org/mathscinet-getitem?mr=#1}{#2}
}
\providecommand{\href}[2]{#2}
\begin{thebibliography}{CDLSJ12}

\bibitem[Bor65]{Borisov:1965wf}
Yurii~F. Borisov, \emph{{C$^{1,\,\alpha }$-isometric immersions of Riemannian
  spaces}}, Dokl. Akad. Nauk SSSR (N.S.) \textbf{163} (1965), 11--13.

\bibitem[Bor04]{Borisov:2004wo}
\bysame, \emph{{Irregular $C^{1,\beta}$-Surfaces with an Analytic Metric}},
  Siberian Mathematical Journal \textbf{45} (2004), no.~1, 19--52.

\bibitem[Cai61]{Cairns1961}
Stewart~S. Cairns, \emph{A simple triangulation method for smooth manifolds},
  Bull. Amer. Math. Soc. \textbf{67} (1961), 389--390. \MR{149491}

\bibitem[Car28]{Cartan1927}
{\'E}lie Cartan, \emph{Sur la possibilit{\'e} de plonger un espace riemannien
  donn{\'e} dans un espace euclidien.}, Ann. Soc. Polon. Math. \textbf{6}
  (1928), 1--7 (French).

\bibitem[CDLSJ12]{CoDeSz2012}
Sergio Conti, Camillo De~Lellis, and L{\'a}szl{\'o} Sz{\'e}kelyhidi~Jr,
  \emph{{$h$-principle and rigidity for $C^{1,\alpha}$ isometric embeddings}},
  Nonlinear Partial Differential Equations: The Abel Symposium 2010 (H~Holden
  and K~H Karlsen, eds.), Springer, 2012, pp.~83--116.

\bibitem[CI20]{CaoIn2020}
Wentao Cao and Dominik Inauen, \emph{{Rigidity and flexibility of isometric
  extension}}, arXiv (2020).

\bibitem[CS19]{CaoSze2019}
Wentao Cao and L\'{a}szl\'{o} Sz\'{e}kelyhidi, Jr., \emph{{$C^{1,\alpha}$}
  isometric extensions}, Comm. Partial Differential Equations \textbf{44}
  (2019), no.~7, 613--636. \MR{3949128}

\bibitem[CS22]{CaoSze2022}
\bysame, \emph{Global {N}ash-{K}uiper theorem for compact manifolds}, J.
  Differential Geom. \textbf{122} (2022), no.~1, 35--68. \MR{4507470}

\bibitem[DLI20]{DI2020}
Camillo De~Lellis and Dominik Inauen, \emph{{$C^{1, \alpha}$} isometric
  embeddings of polar caps}, Adv. Math. \textbf{363} (2020), 106996, 39.
  \MR{4054053}

\bibitem[DLISJ18]{DISz2018}
Camillo De~Lellis, Dominik Inauen, and L{\'a}szl{\'o} Sz{\'e}kelyhidi~Jr,
  \emph{{A Nash-Kuiper theorem for $C^{1,1/5-\delta}$ immersions of surfaces in
  3 dimensions}}, Revista Matem{\'a}tica Iberoamericana \textbf{math.DG}
  (2018).

\bibitem[DSJ17]{DaneriSz2016}
Sara Daneri and L{\'a}szl{\'o} Sz{\'e}kelyhidi~Jr, \emph{{Non-uniqueness and
  h-principle for H\"older-continuous weak solutions of the Euler equations}},
  Arch. Rational Mech. Anal. \textbf{224} (2017), no.~2, 471--514.

\bibitem[G\"89]{Gunther1989}
Matthias G\"{u}nther, \emph{On the perturbation problem associated to isometric
  embeddings of {R}iemannian manifolds}, Ann. Global Anal. Geom. \textbf{7}
  (1989), no.~1, 69--77. \MR{1029846}

\bibitem[G\"91]{Gunther1991}
\bysame, \emph{Isometric embeddings of {R}iemannian manifolds}, Proceedings of
  the {I}nternational {C}ongress of {M}athematicians, {V}ol. {I}, {II}
  ({K}yoto, 1990), Math. Soc. Japan, Tokyo, 1991, pp.~1137--1143. \MR{1159298}

\bibitem[GR70]{GroRoh1970}
M.~L. Gromov and V.~A. Rohlin, \emph{Imbeddings and immersions in {R}iemannian
  geometry}, Uspehi Mat. Nauk \textbf{25} (1970), no.~5 (155), 3--62.
  \MR{0290390}

\bibitem[Gro73]{Gromov1973}
Mikhail Gromov, \emph{Convex integration of differential relations. {I}}, Izv.
  Akad. Nauk SSSR Ser. Mat. \textbf{37} (1973), 329--343. \MR{0413206}

\bibitem[Gro86]{Gromov}
\bysame, \emph{{Partial differential relations}}, Ergebnisse der Mathematik und
  ihrer Grenzgebiete, vol.~9, Springer Verlag, Berlin, 1986.

\bibitem[Gro17]{Gromov:2015tua}
Misha Gromov, \emph{{Geometric, algebraic, and analytic descendants of Nash
  isometric embedding theorems}}, Bull. Amer. Math. Soc. (N.S.) \textbf{54}
  (2017), no.~2, 173--245.

\bibitem[HH06]{HanHongBook}
Qing Han and Jia-Xing Hong, \emph{Isometric embedding of {R}iemannian manifolds
  in {E}uclidean spaces}, Mathematical Surveys and Monographs, vol. 130,
  American Mathematical Society, Providence, RI, 2006. \MR{2261749}

\bibitem[HW17]{HuWa2017}
Norbert Hungerb\"{u}hler and Micha Wasem, \emph{{The one-sided isometric
  extension problem}}, Results Math. \textbf{71} (2017), no.~3-4, 749--781.

\bibitem[Jan27]{Janet1927}
M.~Janet, \emph{Sur la possibilit{\'e} de plonger un espace riemannien
  donn{\'e} dans un espace euclidien.}, Ann. Soc. Polon. Math. \textbf{5}
  (1927), 38--43 (French).

\bibitem[K\"78]{Kallen1978}
Anders K\"{a}ll\'{e}n, \emph{Isometric embedding of a smooth compact manifold
  with a metric of low regularity}, Ark. Mat. \textbf{16} (1978), no.~1,
  29--50. \MR{499136}

\bibitem[Kui55]{Kuiper:1955th}
Nicolaas~H Kuiper, \emph{{On $C^1$-isometric imbeddings. I, II}}, Nederl. Akad.
  Wetensch. Indag. Math. \textbf{17} (1955), 545--556, 683--689.

\bibitem[Lew22]{lewicka2022}
Marta Lewicka, \emph{The monge-ampere system: convex integration in arbitrary
  dimension and codimension}, 2022.

\bibitem[Nas54]{Nash:1954vt}
John Nash, \emph{{$C^1$ isometric imbeddings}}, Ann. of Math. (2) \textbf{60}
  (1954), no.~3, 383--396.

\bibitem[Nas56]{Nash:1956ty}
\bysame, \emph{{The imbedding problem for Riemannian manifolds}}, Ann. of Math.
  (2) (1956), 20--63.

\bibitem[SJ14]{Szek2014:LN}
L{\'a}szl{\'o} Sz{\'e}kelyhidi~Jr, \emph{{From Isometric Embeddings to
  Turbulence}}, HCDTE Lecture Notes. Part II. Nonlinear Hyperbolic PDEs,
  Dispersive and Transport Equations, American Institute of Mathematical
  Sciences, 2014, pp.~1--66.

\bibitem[Whi44]{Whitney1944}
Hassler Whitney, \emph{The self-intersections of a smooth {$n$}-manifold in
  {$2n$}-space}, Ann. of Math. (2) \textbf{45} (1944), 220--246. \MR{10274}

\end{thebibliography}


\end{document}